\documentclass[10pt,reqno]{amsart}

\usepackage[hmargin=3cm, vmargin=3cm]{geometry}

\newcounter{mnotecount}[section]

\newcommand{\rmnote}[1]{}
\usepackage{amsfonts,latexsym}
\usepackage{graphics}
\usepackage{amsmath,amssymb,amsthm,amsfonts,graphicx,color}
\usepackage{amssymb}
\usepackage{amsfonts,amsthm,amsmath,latexsym,bm,graphics}
\def\theequation{\thesection.\@arabic\c@equation}
\makeatother

\newtheorem{theorem}{Theorem}[section]
\def\bs{\begin{split}}
\def\es{\end{split}}
\newcommand{\R}{\mathbb{R}}

\renewcommand{\theequation}{\thesection.\arabic{equation}}
\newtheorem{proposition}  {Proposition}[section]

\begin{document}
\title[Nondegeneracy of nonradial solutions]{Nondegeneracy of nonradial sign-changing solutions to the nonlinear Schr\"odinger equations}

\author{Weiwei Ao}
\address{\noindent W. Ao -Department of Mathematics, University of British Columbia, Vancouver, BC V6T1Z2, Canada }\email{wwao@math.ubc.ca}

\author{Monica Musso}
\address{\noindent M. Musso -
Departamento de Matem\'atica,
Pontificia Universidad Catolica de Chile, Avda. Vicu\~na Mackenna
4860, Macul, Chile}
\email{mmusso@mat.puc.cl}

\author{Juncheng Wei}
\address{\noindent J. Wei -Department of Mathematics, University of British Columbia, Vancouver, BC V6T1Z2, Canada }\email{jcwei@math.ubc.ca}

\thanks{ The research of the second  author
has been partly supported by Fondecyt Grant 1120151 and Millennium Nucleus Center for Analysis of PDE, NC130017. The research of the third author is partially supported by NSERC of Canada.}

\date{}\maketitle

\begin{abstract}
We prove that the non-radial sign-changing solutions to the nonlinear Schr\"odinger equation
\begin{equation*}
\Delta u-u+|u|^{p-1}u=0 \mbox{ in }\R^N, \quad u \in H^1 (\R^N )
\end{equation*}
constructed by Musso, Pacard and Wei \cite{MPW} is non-degenerate. This provides the first example of non-degenerate sign-changing solution with finite energy to the
above nonlinear Schr\"odinger equation.
\end{abstract}

\setcounter{equation}{0}
\section{Introduction and statement of main results}\label{sec1}
In this paper, we consider the nonlinear semilinear elliptic equation:
\begin{equation}\label{eq1}
\Delta u-u+|u|^{p-1}u=0 \mbox{ in }\R^N, \ u\in H^1(\R^N),
\end{equation}
where $p$ satisfies $1<p<\infty$ when $N=2$, and $1<p<\frac{N+2}{N-2}$ if $N\geq 3$. Equation (\ref{eq1}) arises in various models in physics, mathematical physics and biology. In particular, the study of standing waves for the nonlinear Klein-Gordon or Schr\"odinger equations reduces to (\ref{eq1}). We refer to the papers of Berestycki-Lions \cite{BL}, \cite{BL1}, and Bartsch-Willem \cite{BW} for further references and motivations.

\medskip
Denote the set of non-zero finite energy solutions to (\ref{eq1}) by
\begin{equation*}
\Sigma:=\{u\in H^1(\R^N): \Delta u-u+|u|^{p-1}u=0\}.
\end{equation*}

If $u \in \Sigma$ and $u>0$, the classical result of Gidas, Ni and Nirenberg \cite{GNN} asserts that $u$ is then radially symmetric.  Indeed, it is known (\cite{K, kz}) that
there exists a unique radially symmetric (in
fact radially decreasing) positive solution of
\[
\Delta w - w + w^{p} \,  = 0 \quad {\mbox {in}} \quad \R^N,
\]
which tends to $0$ as $|x|$ tends to $\infty$. All the other
positive solutions to (\ref{eq1}) belonging to $\Sigma$  are translations of $w$. 

\medskip
Let $L_0$ be the linearized operator around $w$, defined by
\begin{equation} \label{L0}
L_0 : = \Delta  - 1 + p \, w^{p-1} \, .
\end{equation}
The natural invariance of problem (\ref{eq1}) under the group of isometries in $\R^N$ reduces to the fact that the functions
\begin{equation} \label{zeta}
\partial_{x_1} w  , \ldots , \partial_{x_N} w \,  ,
\end{equation}
naturally belong to the kernel of the operator $L_0$. 
The solution $w $ is {\em nondegenerate} in the sense
that $L^\infty$-kernel of the operator $L_0$ is spanned by the functions
given in (\ref{zeta}). See \cite{nitakagi}.

\medskip
No other example of {\em nondegenerate} solution to (\ref{eq1}) in $\Sigma$ is known. The purpose of this paper is to provide the first example of {\em nondegenerate} solution
to (\ref{eq1}) in $\Sigma$ other than $w$.

\medskip

Concerning existence of other solutions to (\ref{eq1}) in $\Sigma$, several results are available in the literature.  Berestychi-Lions \cite{BL}, \cite{BL1} and Struwe \cite{St} have obtained the existence of infinitely many radially symmetric sign-changing solutions. The proofs of these results are based on the use  of variational methods. The existence of nonradial sign-changing solutions was first proved by Bartsch-Willem \cite{BW} in dimension $N=4$ and $N\geq 6$. Also in this case, the resut is proved by means of variational methods, and their key idea is to look for solutions invariant under the action $O(2)
\times O(N-2)$ in order to recover some compactness property. Later on, the result was generalized by Lorca-Ubilla \cite{LU} to handle the $N=5$ dimensional case. 
Besides the symmetry property of the solutions, the mentioned results do not provide any other qualitative properties for the solutions. A different approach, and a different construction, have been developed recently in \cite{MPW} and \cite{AMPW}, where new types of non radial sign-changing finite-energy solutions to (\ref{eq1})  are built and a detailed description of these solutions is provided.

\medskip
The main purpose of this paper is to prove that the solutions constructed by Musso-Pacard-Wei in \cite{MPW} are rigid, up to the transformations of the equation. In other words, these solutions are {\em non-degenerate}, in the sense of the definition introduced by Duyckaerts-Kenig-Merle \cite{DKM}. 

\medskip
To explain what being a {\em non degenerate} solution for a given $u \in \Sigma $ means, we recall all the possible invariances of Equation (\ref{eq1}). 
We have that Equation (\ref{eq1}) is invariant under the following two transformations:

$(1)$ (translation): \ \ If $u\in \Sigma$,then $u(x+a)\in \Sigma$, $\forall a\in \R^N$;

$(2)$ (rotation): \ \ If $u\in \Sigma$, then $u(Px)\in \Sigma$, where $P\in O_N$ and $O_N$ is the classical orthogonal group.

If $u\in \Sigma$, we denote
\begin{equation}\label{eq3}
L_u=\Delta-1+p|u|^{p-1}
\end{equation}
the linearized operator around $u$. Define the null space of $L_u$
\begin{equation}
\mathcal{Z}_u=\{f\in H^1(\R^N): L_uf=0\}.
\end{equation}
If we denote by $\mathcal{M}$ the group of isometries of $H^1(\R^N)$ generated by the previous two transformations, then the elements in $\mathcal{Z}_u$ generated by the family of transformation $\mathcal{M}$ define the following vector space:
\begin{equation}\label{eq4}
\tilde{\mathcal{Z}}_u=\operatorname{span}\left\{\begin{array}{l}\partial_{x_j}u, \ 1\leq j\leq N\\
(x_j\partial_{x_k}-x_k\partial_{x_j})u, \ 1\leq j< k\leq N
\end{array}
\right\} .
\end{equation}

A solution $u$ of (\ref{eq1}) is {\em non-degenerate} if
\begin{equation}\label{eq2}
\mathcal{Z}_u=\tilde{\mathcal{Z}}_u,
\end{equation}
see \cite{DKM}.

As we already mentioned, so far the only non degeneracy example of $u\in\Sigma$ is the positive solution $w$. In fact, in this case
$$
(x_j\partial_{x_k}-x_k\partial_{x_j}) w = 0 , \quad \forall   \quad 1\leq j< k\leq N.
$$
 and hence 
$$
{\mathcal Z}_w = \operatorname{span}\left\{ \partial_{x_j}w, \ 1\leq j\leq N
\right\} .
$$
The proof of the non-degeneracy of $w$ relies heavily on the radial symmetry of $w$. For  nonradial solutions, the strategy to prove non-degeneracy for the radial case no longer works out. So a new strategy is needed for the non radially symmetric solutions.

A similar problem has arised in the study of non radial sign-changing finite-energy solutions for the Yamabe typw problem
$$
\Delta u + |u|^{4\over N-2} u = 0 , \quad {\mbox {in}} \quad \R^n, \quad u \in H^1 (\R^N ),
$$
for $N\geq 3$.   In \cite{MW} Musso-Wei  first introduce some new ideas to deal with the  non degeneracy for a  given non radial sign changing solution to the above problem. Indeed, they  successfully analyze the non degeneracy of some nonradial solutions to the Yamabe problem, that were previously constructed in \cite{dmpp}. See also \cite{dmpp1} for other constructions. In this paper, we will adopt the idea developed in \cite{MW} to analyze the non-degeneracy of the solutions of (\ref{eq1}) constructed by Musso-Pacard-Wei \cite{MPW}.

\medskip

The main result of this paper can be stated as follows:
\begin{theorem}\label{theorem1}
There exists a sequence of non radial sign-changing solutions to (\ref{eq1}) with arbitrary large energy, and each solution is {\em non degenerate} in the sense of (\ref{eq2}).
\end{theorem}

\medskip
We believe that the property of being non degenerate for the solutions in Theorem \ref{theorem1} can be used to obtain new type of constructions for sign changing solutions to (\ref{eq1}), or related problems in bounded domain with Dirichlet or Neumann boundary conditions.

The paper is organized as follows. In Section \ref{sec2}, we introduce the solutions constructed by Musso, Pacard and Wei in \cite{MPW}. In Section \ref{sec3}, we sketch the main steps in the proof of Theorem \ref{theorem1}. Section \ref{sec4} to Section \ref{sec7} are devoted to the proof of Theorem \ref{theorem1}.

\setcounter{equation}{0}
\section{Description of the solutions}\label{sec2}
In this section, we describe the solutions $u_\ell$ constructed in \cite{MPW}, and recall some properties that will be useful for later purpose.

To descript the solutions, let us introduce some notations first. The canonical basis of $\R^N$ will be denoted by
\begin{equation}
{\bf e}_1=(1,0,\cdots,0), \ {\bf e}_2=(0,1,0,\cdots,0), \cdots, {\bf e}_N=(0,\cdots,0,1),
\end{equation}
and fix an integer $k\geq 7$.

First we are given two positive integers  $m,n$ and two positive real numbers $\ell, \bar{\ell}$ which are related by
\begin{equation*}
2\sin\frac{\pi}{k}m\ell=(2n-1)\bar{\ell}.
\end{equation*}
We consider the {\it   inner polygon}  which is the regular polygon in $\R^2\times \{0\}\subset \R^N$, with $k $ edges and whose vertices are given by the orbit of the point
\begin{equation*}
y_1=\frac{\bar{\ell}}{2\sin\frac{\pi}{k}}{\bf e}_1\in \R^N,
\end{equation*}
under the action of the group generated by $R_k$. Here $R_k\in O(2)\times O(N-2)$ is the rotation of angle $\frac{2\pi}{k} $ in the $(x_1,x_2)$ plane. By construction, the edges of this polygon has length $\ell$. Define now the {\it outer polygon} which is a regular polygon with $k$ edges and  whose vertices are the orbit of the point
\begin{equation*}
y_{m+1}=y_1+m\ell {\bf e}_1,
\end{equation*}
under the group generated by $R_k$.  By construction, the distance between the points $y_1$ and $y_{m+1}$ is equal to $m\ell$ and we denote by $y_j$ for $j=2,\cdots,m$ the points evenly distributed on the segment between these two points. Namely,
\begin{equation*}
y_j=y_1+(j-1)\ell{\bf e}_1 \mbox{ for } j=2,\cdots,m.
\end{equation*}
The edges of the outer polygon have length $2n\bar{\ell}$ and we distribute evenly points $y_j$, $j=m+2,\cdots,m+2n$ along this segment. More precisely, if we define
\begin{equation*}
{\bf t}=-\sin\frac{\pi}{k}{\bf e}_1+\cos\frac{\pi}{k}{\bf e}_2\in \R^N,
\end{equation*}
then, the points $y_j$ are given by
\begin{equation*}
y_j=y_{m+1}+(j-m-1)\bar{\ell}{\bf t} \mbox{ for }j=m+2,\cdots,m+2n.
\end{equation*}
We also denote by
\begin{equation*}
z_h=y_j, \mbox{ for } h=1,\cdots,2n-1, \ h=j-m-1.
\end{equation*}
Observe that, by construction
\begin{equation*}
R_ky_{m+1}=y_{m+1}+2n\ell{\bf t}.
\end{equation*}
In \cite{MPW}, they constructed solutions that can be described as follows:
\begin{equation*}
u_{\ell}(x)=U+\phi(x),
\end{equation*}
where
\begin{equation*}
U=\sum_{i=0}^{k-1}(\sum_{j=1}^{m+1}w(x-R_k^iy_j)+\sum_{j=m+2}^{m+2n}(-1)^{j-m-1}w(x-R_k^iy_j)),
\end{equation*}
and $\phi=o(1)$ as $\ell \to \infty$. The function $w$ is the unique solution of the following equation:
\begin{equation*}
\left\{\begin{array}{l}
\Delta u-u+u^p=0,\ u>0 \mbox{ in } \R^N\\
\max_{x\in \R^N}u(x)=u(0)
\end{array}
\right.
\end{equation*}
The solutions they constructed enjoy the following invariance
\begin{equation*}
u(x)=u(Rx), \mbox{ for }R\in \{I_2\}\times O(N-2),
\end{equation*}
and
\begin{equation*}
u(R_kx)=u(x) \mbox{ and }u(\Gamma x)=u(x)
\end{equation*}
where $R_k\in O(2)\times O(N-2)$ is the rotation of angle $\frac{2\pi}{k} $ in the $(x_1,x_2)$ plane and $\Gamma \in O(2)\times I(N-2)$ is the symmetry with respect to the hyperplane $x_2=0$.

The numbers $m,n,\ell,\bar{\ell}$ are related by the following two equations:
\begin{equation}
\Psi(\ell)=2\sin\frac{\pi}{k}\Psi(\bar{\ell}),
\end{equation}
and
\begin{equation}
2\sin\frac{\pi}{k}m\ell=(2n-1)\bar{\ell}.
\end{equation}
where
\begin{equation}\label{psi}
\Psi(s)=-\int w(x-s{\bf e}) div(w^p(x){\bf e})dx
\end{equation}
and ${\bf e}\in \R^N$ is any unit vector. The definition of $\Psi$ is independent of ${\bf e}$.

In \cite{MPW}, they proved that $m,n=O(\ell)$ and
\begin{equation*}
\bar{\ell}=\ell+\ln(2\sin\frac{\pi}{k})+O(\frac{1}{\ell}),
\end{equation*}
\begin{equation*}
-(\log\Psi)'(s)=1+\frac{N-1}{2s}+O(\frac{1}{s^2}).
\end{equation*}

Since the solution the constructed has the form $u_\ell=U+\phi$, in terms of $\phi$, the equation (\ref{eq1}) gets rewritten as
\begin{equation}
\Delta \phi-\phi+p|U|^{p-1}\phi+E+N(\phi)=0
\end{equation}
where
\begin{equation}
E=\Delta U-U+|U|^{p-1}U,
\end{equation}
and
\begin{equation}
N(\phi)=|U+\phi|^{p-1}(U+\phi)-|U|^{p-1}U-p|U|^{p-1}\phi.
\end{equation}

Let us fix a number $-1<\eta<0$, and define the weighted norm
\begin{equation}\label{norm}
\|h\|_*=\sup_{x\in \R^N}|(\sum_{y\in \Pi}e^{\eta|x-y|})^{-1}h(x)|
\end{equation}
where
\begin{equation}
\Pi=\cup_{i=0}^{k-1}\{R_k^iy_j : j=1,\cdots,m+2n\}.
\end{equation}

In \cite{MPW}, it is proved that there exists $\ell_0>0$ and $\xi>0$, such that for $\ell>\ell_0$, $\phi$ satisfies the following estimate (Proposition 4.1 in \cite{MPW}):
\begin{equation}
\|\phi\|_*\leq Ce^{-\frac{1+\xi}{2}\ell}.
\end{equation}

Let us now define the following functions:
\begin{equation}
\pi_\alpha(x)=\frac{\partial }{\partial x_\alpha}\phi(x) \mbox{ for }\alpha=1,\cdots,N.
\end{equation}

Then the function $\pi_\alpha$ can be described as
\begin{proposition}\label{pro1}
The function $\pi_\alpha$ satisfies the following estimates:
\begin{equation}\label{pialpha}
\|\pi_\alpha\|_*\leq Ce^{-\frac{1+\xi}{2}\ell},
\end{equation}
for some positive constant $C$ and $\xi$ independent of $\ell$ large.
 \end{proposition}

Recall that problem (\ref{eq1}) is invariant under the two transformations mentioned in Section \ref{sec1}: translation, rotation. These invariance will be reflected in the element of the kernel of the linearized operator
\begin{equation}\label{defLLL}
L(\psi):=\Delta \psi-\psi+p|u_\ell|^{p-1}\psi
\end{equation}
which is the linearized equation associated to (\ref{eq1}) around $u_\ell$.

From now on, we will drop $\ell$ in $u_\ell$ for simplicity. Let us now introduce the following $3N-3$ functions:
\begin{equation}\label{translation}
z_\alpha(x)=\frac{\partial }{\partial x_\alpha}u(x), \mbox{ for }\alpha=1,\cdots,N,
\end{equation}
and
\begin{equation}\label{12rotation}
z_{N+1}(x)=x_1\frac{\partial }{\partial x_2}u(x)-x_2\frac{\partial }{\partial x_1}u(x).
\end{equation}
Furthermore, for $\alpha=3,\cdots,N$
\begin{equation}\label{13rotation}
z_{N+\alpha-1}(x)=x_1z_\alpha-x_\alpha z_1,\ z_{2N+\alpha-3}(x)=x_2z_\alpha-x_\alpha z_2.
\end{equation}
Observe that  the functions defined in (\ref{translation}) are related to the invariance of (\ref{eq1}) under translation, while the functions defined in (\ref{12rotation}) and (\ref{13rotation}) are related to the invariance of (\ref{eq1}) under the rotation in $(x_1,x_2)$ plane, $(x_1,x_\alpha)$ plane and $(x_2,x_\alpha)$ plane respectively.

 The invariance of problem (\ref{eq1}) under translation and rotation gives that the set $\tilde{Z}_u$ (introduced in (\ref{eq4}))associated to the linear operator $L$ introduced in (\ref{defLLL}) has dimension at least $3N-3$, since
 \begin{equation}
 L(z_\alpha)=0, \ \alpha=1,\cdots,3N-3.
 \end{equation}
 We will show that these functions are the only bounded elements of the kernel of the operator $L$.

 \setcounter{equation}{0}
\section{Scheme of the proof}\label{sec3}
Let $\varphi$ be a bounded function satisfying $L(\varphi)=0$, where $L$ is the linear operator defined by (\ref{eq3}). We write our function $\varphi$ as
\begin{equation}
\varphi(x)=\sum_{\alpha=1}^{3N-3}a_\alpha z_{\alpha}(x)+\tilde{\varphi}(x),
\end{equation}
where the functions $z_\alpha(x)$ are defined by (\ref{translation}), (\ref{12rotation}) and (\ref{13rotation}) respectively, while the constant $a_\alpha$ are chosen such that
\begin{equation}\label{ortho1}
\int z_\alpha \tilde{\varphi}=0, \alpha=1,\cdots,3N-3.
\end{equation}
Observe that $L(\tilde{\varphi})=0$. Our aim is to show that, if $\tilde{\varphi}$ is bounded, then $\tilde{\varphi}=0$. We introduce the following functions:

For $j=1,\cdots, m+1$, $i=0,\cdots,k-1$
\begin{equation*}
\tilde{Z}_{j,1}^i=R_k^i\cdot \nabla w(x-R_k^iy_j),
\end{equation*}
\begin{equation*} \tilde{Z}_{j,2}^i=R_k^{i,\perp}\cdot \nabla w(x-R_k^i y_j),
\end{equation*}
and
\begin{equation*}
\tilde{Z}_{j,\alpha}^i=\frac{\partial }{\partial x_{\alpha}}w(x-R_k^iy_j)
\end{equation*}
for $\alpha=3,\cdots,N$.

Moreover, we define for $j=m+2,\cdots, m+2n$, $i=0,\cdots,k-1$
\begin{equation*}
\tilde{Z}_{j,1}^i=(-1)^{j-m-1} {\bf t}_i\cdot \nabla w(x-R_k^iy_j),
\end{equation*}
\begin{equation*}
\tilde{Z}_{j,2}^i=(-1)^{j-m-1}{\bf n}_i\cdot \nabla w(x-R_k^iy_j),
\end{equation*}
and
\begin{equation*}
\tilde{Z}_{j,\alpha}^i=(-1)^{j-m-1}\frac{\partial }{\partial x_\alpha}w(x-R_k^iy_j)
\end{equation*}
for $\alpha=3,\cdots,N$,
where
\begin{equation*}
R_k^i=(\cos\frac{2\pi i}{k},\sin\frac{2\pi i}{k},0),
\end{equation*}
\begin{equation*}
 R_k^{i,\perp}=(\sin\frac{2\pi i}{k},-\cos\frac{2\pi i}{k},0),
\end{equation*}
\begin{equation*}
{\bf t}_i=(-\sin(\frac{2\pi i}{k}+\frac{\pi}{k}),\cos(\frac{2\pi i}{k}+\frac{\pi}{k}),0),
\end{equation*}
\begin{equation*}
 {\bf n}_i=(\cos(\frac{2\pi i}{k}+\frac{\pi}{k}), \sin(\frac{2\pi i}{k}+\frac{\pi }{k}),0).
\end{equation*}
In the following, we will denote by $\theta_i=\frac{2\pi i}{k}$. Then we can write that
\begin{eqnarray*}
z_1(x)&&=\pi_1+\frac{\partial U}{\partial x_1}\\
&&=\pi_1+\sum_{i=1}^{N-1}\Big(\sum_{j=1}^{m+1}(\cos\frac{2\pi i}{k}\tilde{Z}_{j,1}^i+\sin\frac{2\pi i}{k}\tilde{Z}_{j,2}^i)\\
&&-\sum_{j=m+2}^{2n+m}(\sin(\frac{2\pi i}{k}+\frac{\pi}{k})\tilde{Z}_{j,1}^i-\cos(\frac{2\pi i}{k}+\frac{\pi}{k})\tilde{Z}_{j,2}^i)\Big),
\end{eqnarray*}
\begin{eqnarray*}
z_2(x)&&=\pi_2+\frac{\partial U}{\partial x_2}\\
&&=\pi_2+\sum_{i=0}^{k-1}\Big(\sum_{j=1}^{m+1}(\sin\frac{2\pi i}{k}\tilde{Z}_{j,1}^i-\cos\frac{2\pi i}{k}\tilde{Z}_{j,2}^i)\\
&&+\sum_{j=m+2}^{2n+m}(\cos(\frac{2\pi i}{k}+\frac{\pi}{k})\tilde{Z}_{j,1}^i+\sin(\frac{2\pi i}{k}+\frac{\pi}{k})\tilde{Z}_{j,2}^i)\Big),
\end{eqnarray*}
and
\begin{equation*}
z_\alpha=\frac{\partial u}{\partial x_\alpha}=\pi_\alpha+\sum_{i=0}^{k-1}(\sum_{j=1}^{m+2n}\tilde{Z}_{j,\alpha}^i),
\end{equation*}
for $\alpha=3,\cdots,N$.
Furthermore
\begin{eqnarray*}
z_{N+1}&&=x_1z_2-x_2z_1\\
&&=x_1\pi_2-x_2\pi_1+x_1\frac{\partial U}{\partial x_2}-x_2\frac{\partial U}{\partial x_1}\\
&&=x_1\pi_2-x_2\pi_1+\sum_{i=0}^{k-1}\Big(\sum_{j=1}^{m+1}|y_j|(\cos \theta_i\frac{\partial }{\partial x_2}-\sin\theta_i\frac{\partial }{\partial x_1})w(x-R_k^iy_j)\\
&&+\sum_{j=m+2}^{2n+m}(R_k^iy_j\cdot{\bf n}_i\tilde{Z}_{j,1}^i-R_k^iy_j\cdot {\bf t}_i \tilde{Z}_{j,2}^i)
\Big)
\end{eqnarray*}
and for $\alpha=3,\cdots,N$,
\begin{eqnarray*}
z_{N+\alpha-1}&&=x_1\pi_\alpha-x_\alpha\pi_1+x_1\frac{\partial U}{\partial x_\alpha}-x_\alpha\frac{\partial U}{\partial x_1}\\
&&=x_1\pi_\alpha-x_\alpha\pi_1+\sum_{i=0}^{k-1}\Big(\sum_{j=1}^{m+1}|y_j|\cos \theta_i\tilde{Z}_{j,\alpha}^i\\
&&+\sum_{j=m+2}^{2n+m}(R_k^iy_j\cdot{\bf n}_i\cos(\theta_i+\frac{\pi}{k})-R_k^iy_j\cdot{\bf t}_i\sin(\theta_i+\frac{\pi}{k}))\tilde{Z}_{j,\alpha}^i\Big),
\end{eqnarray*}
and
\begin{eqnarray*}
z_{2N+\alpha-3}&&=x_2\pi_\alpha-x_\alpha\pi_2+x_2\frac{\partial U}{\partial x_\alpha}-x_\alpha\frac{\partial U}{\partial x_2}\\
&&=x_2\pi_\alpha-x_\alpha\pi_2+\sum_{i=0}^{k-1}\Big(\sum_{j=1}^{m+1}|y_j|\sin \theta_i\tilde{Z}_{j,\alpha}^i\\
&&+\sum_{j=m+2}^{2n+m}(R_k^iy_j\cdot{\bf n}_i\sin(\theta_i+\frac{\pi}{k})+R_k^iy_j\cdot{\bf t}_i\cos(\theta_i+\frac{\pi}{k}))\tilde{Z}_{j,\alpha}^i\Big).
\end{eqnarray*}

\noindent
Let us now define the following functions. For $i=0,\cdots,k-1$
\begin{equation}
Z_{1,1}^i=\cos\theta_i(\frac{\partial w(x-R_k^iy_1)}{\partial x_1}+\frac{\pi_1}{k})+\sin\theta_i(\frac{\partial w(x-R_k^iy_1)}{x_2}+\frac{\pi_2}{k}),
\end{equation}
and
\begin{equation}
Z_{1,2}^i=\sin\theta_i(\frac{\partial w(x-R_k^iy_1)}{\partial x_1}+\frac{\pi_1}{k})-\cos\theta_i(\frac{\partial w(x-R_k^iy_1)}{x_2}+\frac{\pi_2}{k}),
\end{equation}
and
\begin{equation}
Z_{1,\alpha}^i=\frac{\partial w(x-R_k^iy_1)}{x_\alpha}+\frac{\pi_\alpha}{k}
\end{equation}
for $\alpha=3,\cdots,N$.
Moreover, we define the following functions:
\begin{equation}
Z_{j,\alpha}^i=\tilde{Z}_{j,\alpha}^i \mbox{ for } i=0,\cdots,k-1, j=2,\cdots,2n+m, \alpha=1,\cdots,N.
\end{equation}

\medskip

For the simplicity of notations, we first introduce the following vectors:
\begin{equation*}
Z_{{\bf v},\alpha}=\left(\begin{array}{c}
Z_{1,\alpha}^0\\
\vdots\\
Z_{1,\alpha}^{k-1}\\
Z_{m+1,\alpha}^0\\
\vdots\\
Z_{m+1,\alpha}^{k-1}
\end{array}
\right) , \ Z_{Y_1,\alpha}^i=\left(\begin{array}{c}
Z_{2,\alpha}^i\\
\vdots\\
Z_{m,\alpha}^{i}
\end{array}
\right), \  Z_{Y_2,\alpha}^i=\left(\begin{array}{c}
Z_{{m+2},\alpha}^i\\
\vdots\\
Z_{{2n+m},\alpha}^i
\end{array}
\right),
\end{equation*}
and
\begin{equation*}
{\bf Z}_\alpha=\left(\begin{array}{c}
Z_{{\bf v},\alpha}\\
Z_{Y_1,\alpha}^0\\
\vdots\\
Z_{Y_1,\alpha}^{k-1}\\
Z_{Y_2,\alpha}^0\\
\vdots\\
Z_{Y_2,\alpha}^{k-1}
\end{array}
\right), \ {\bf Z}=\left(\begin{array}{c}
{\bf Z}_1\\
\vdots\\
{\bf Z}_N
\end{array}
\right).
\end{equation*}

With these notations in mind, we write our function $\tilde{\varphi}$ as
\begin{equation*}
\tilde{\varphi}=\sum_{\alpha=1}^N{\bf c}_\alpha\cdot {\bf Z}_\alpha+\varphi^\perp(x)
\end{equation*}
where ${\bf c}_\alpha=\left(\begin{array}{c}c_{\bf{v},\alpha}\\ c_{Y_1,\alpha}^0\\ \vdots\\ c_{Y_1,\alpha}^{k-1}\\ c_{Y_2,\alpha}^0\\ \vdots\\ c_{Y_2,\alpha}^{k-1}
\end{array}\right)=\left(\begin{array}{c}c_{1,\alpha}\\ \vdots\\ c_{(m+2n)\times k,\alpha} \end{array}\right)$, $\alpha=1,\cdots,N$ are $N$ vectors in $\R^{(m+2n)\times k}$  defined so that
\begin{equation}
\int Z_{j,\alpha}^i\varphi^\perp=0 \mbox{ for \ all } \alpha=1,\cdots,N, \ j=1,\cdots,m+2n.\end{equation}
Observe that
\begin{equation*}
{\bf c}_\alpha=0 \mbox{ for \ all } \alpha, \mbox{ and }\varphi^\perp=0
\end{equation*}
implies $\tilde{\varphi}=0$. Hence our purpose is to show that all vectors ${\bf c}_\alpha$ and $\varphi^\perp$ are zero. This will be consequence of the following three facts.

\medskip
\noindent
{\bf Fact 1}.   The orthogonality condition (\ref{ortho1}) takes the form
\begin{equation}\label{fact1}
\sum_{\alpha=1}^N{\bf c}_\alpha \cdot \int {\bf Z}_\alpha z_\beta=-\int \varphi^\perp z_\beta
\end{equation}
for $\beta=1,\cdots,3N-3$.

Let us now introduce the following vectors:
\begin{equation*}
{\bf cos}_k=\left(\begin{array}{c}
\cos\theta_0\\
\vdots\\
\cos\theta_{k-1}
\end{array}
\right),\ {\bf sin}_k=\left(\begin{array}{c}
\sin\theta_0\\
\vdots\\
\sin\theta_{k-1}
\end{array}
\right),
\end{equation*}
are two $k$-dimensional vectors and
\begin{equation*}
 {\bf cos (\theta_i)}=\left(\begin{array}{c}
\cos\theta_i\\
\vdots\\
\cos\theta_i
\end{array}
\right),\ \ {\bf sin (\theta_i)}=\left(\begin{array}{c}
\sin\theta_i\\
\vdots\\
\sin\theta_i
\end{array}
\right),
\end{equation*}
are two $(m-1)$-dimensional vectors and
\begin{equation*}
\ {\bf cos (\theta_i+\frac{\pi}{k})}=\left(\begin{array}{c}
\cos(\theta_i+\frac{\pi}{k})\\
\vdots\\
\cos(\theta_i+\frac{\pi}{k})
\end{array}
\right),\ \ {\bf sin (\theta_i+\frac{\pi}{k})}=\left(\begin{array}{c}
\sin(\theta_i+\frac{\pi}{k})\\
\vdots\\
\sin(\theta_i+\frac{\pi}{k})
\end{array}
\right).
\end{equation*}
are two $(2n-1)$-dimensional vectors.

\medskip
\noindent
We have the validity of the following
\begin{proposition}\label{pro301}
The system (\ref{fact1}) reduces to the following $3N-3$ linear conditions on the vectors ${\bf c}_\alpha$:
\begin{equation}\label{eq301}
{\bf c}_1\cdot \left(\begin{array}{c}
{\bf cos}_k\\
{\bf cos}_k\\
{\bf cos(\theta_0)}\\
\vdots\\
{\bf cos(\theta_{k-1})}\\
-{\bf sin(\theta_0+\frac{\pi}{k}) }\\
\vdots\\
-{\bf sin(\theta_0+\frac{\pi}{k})}
\end{array}
\right)+{\bf c}_2\cdot \left(\begin{array}{c}
{\bf sin}_k\\
{\bf sin}_k\\
{\bf sin(\theta_0)}\\
\vdots\\
{\bf sin(\theta_{k-1})}\\
{\bf cos(\theta_0+\frac{\pi}{k}) }\\
\vdots\\
{\bf cos(\theta_0+\frac{\pi}{k})}
\end{array}
\right)=t_1+O(e^{-\frac{(1+\xi)\ell}{2}})\mathcal{L}_1\left(\begin{array}{c}{\bf c}_1\\ \vdots\\ {\bf c}_N
\end{array}\right)
\end{equation}
\begin{equation}\label{eq302}
{\bf c}_1\cdot \left(\begin{array}{c}
{\bf sin}_k\\
{\bf sin}_k\\
{\bf sin(\theta_0)}\\
\vdots\\
{\bf sin(\theta_{k-1})}\\
{\bf cos(\theta_0+\frac{\pi}{k}) }\\
\vdots\\
{\bf cos(\theta_0+\frac{\pi}{k})}
\end{array}
\right)+{\bf c}_2\cdot \left(\begin{array}{c}
-{\bf cos}_k\\
-{\bf cos}_k\\
-{\bf cos(\theta_0)}\\
\vdots\\
-{\bf cos(\theta_{k-1})}\\
{\bf sin(\theta_0+\frac{\pi}{k}) }\\
\vdots\\
{\bf sin(\theta_0+\frac{\pi}{k})}
\end{array}
\right)=t_2+O(e^{-\frac{(1+\xi)\ell}{2}})\mathcal{L}_2\left(\begin{array}{c}{\bf c}_1\\ \vdots\\ {\bf c}_N
\end{array}\right)
\end{equation}
and for $\alpha=3,\cdots,N$,
\begin{equation}
{\bf c}_\alpha\cdot \left(\begin{array}{c}
{\bf 1}_k\\
{\bf 1}_k\\
{\bf 1}_{m-1}\\
\vdots\\
{\bf 1}_{m-1}\\
{\bf 1}_{2n-1}\\
\vdots\\
{\bf 1}_{2n-1}
\end{array}
\right)=t_\alpha+O(e^{-\frac{(1+\xi)\ell}{2}})\mathcal{L}_\alpha\left(\begin{array}{c}{\bf c}_1\\ \vdots\\ {\bf c}_N
\end{array}\right)
\end{equation}
\begin{equation}\label{eq303}
{\bf c}_1\cdot\left(\begin{array}{c}
{\bf 0}_k\\
{\bf 0}_k\\
{\bf R_k^0y_j\cdot R_k^{0,\perp}}\\
\vdots\\
{\bf R_k^{k-1}y_j\cdot R_k^{k-1,\perp}}\\
{\bf R_k^0z_h\cdot n_0}\\
\vdots\\
{\bf R_k^{k-1}z_h\cdot n_{k-1}}
\end{array}
\right)-{\bf c}_2\cdot\left(\begin{array}{c}
{\bf |y_1|}_k\\
{\bf |y_{m+1}|}_k\\
{\bf R_k^0y_j\cdot R_k^0}\\
\vdots\\
{\bf R_k^{k-1}y_j\cdot R_k^{k-1}}\\
{\bf R_k^0z_h\cdot t_0}\\
\vdots\\
{\bf R_k^{k-1}z_h\cdot t_{k-1}}
\end{array}
\right)=t_{N+1}+O(e^{-\frac{(1+\xi)\ell}{2}})\mathcal{L}_{N+1}\left(\begin{array}{c}{\bf c}_1\\ \vdots\\ {\bf c}_N
\end{array}\right)
\end{equation}
and for $\alpha=3,\cdots,N$
\begin{equation}\label{eq304}
{\bf c}_\alpha\cdot \left(\begin{array}{c}
|y_1|{\bf cos}_k\\
|y_{m+1}|{\bf cos}_k\\
{\bf R_k^0y_j\cdot e_1}\\
\vdots\\
{\bf R_k^{k-1}y_j\cdot e_1}\\
{\bf R_k^0z_h\cdot e_1}\\
\vdots\\
{\bf R_k^{k-1}z_h\cdot e_1}
\end{array}
\right)=t_{N+\alpha-1}+O(e^{-\frac{(1+\xi)\ell}{2}})\mathcal{L}_{N+\alpha-1}\left(\begin{array}{c}{\bf c}_1\\ \vdots\\ {\bf c}_N
\end{array}\right)
\end{equation}
\begin{equation}\label{eq305}
{\bf c}_\alpha\cdot \left(\begin{array}{c}
|y_1|{\bf sin}_k\\
|y_{m+1|}{\bf sin}_k\\
{\bf R_k^0y_j\cdot e_2}\\
\vdots\\
{\bf R_k^{k-1}y_j\cdot e_2}\\
{\bf R_k^0z_h\cdot e_2}\\
\vdots\\
{\bf R_k^{k-1}z_h\cdot e_2}
\end{array}
\right)=t_{2N+\alpha-3}+O(e^{-\frac{(1+\xi)\ell}{2}})\mathcal{L}_{2N+\alpha-3}\left(\begin{array}{c}{\bf c}_1\\ \vdots\\ {\bf c}_N
\end{array}\right).
\end{equation}
In the above expansions, $\left(\begin{array}{c}t_1\\ \vdots\\
t_{3N-3}\end{array}\right)$ is a fixed vector with
\begin{equation*}
\|\left(\begin{array}{c}t_1\\ \vdots\\
t_{3N-3}\end{array}\right)\|\leq \ell^\tau\|\varphi^\perp\|_*
\end{equation*}
for some positive constant $\tau$ independent of $\ell $ large and $\mathcal{L}_i: \R^{(2n+m)\times k}\to \R$ are linear functions whose coefficients are constants uniformly bounded as $\ell\to \infty$.
\end{proposition}

\begin{proof}
Let us consider (\ref{fact1}) with $\beta=1$, that is

\begin{equation}
\sum_{\alpha=1}^N {\bf c}_\alpha \cdot \int {\bf Z}_\alpha z_1=-\int \varphi^\perp z_1.
\end{equation}

First we write $t_1=-\frac{\int \varphi^\perp z_1}{\int (\frac{\partial w(x)}{\partial x_1})^2}$. A straightforward computation gives that $|t_1|\leq C(2n+m)\times k\|\varphi^\perp\|_*\leq \ell^\tau \|\varphi^\perp\|_*$ for a certain constant $\tau$ independent of $\ell$ large. Second, by direct computation, we have for $j=1,\cdots,m+1$
\begin{eqnarray*}
\int Z_{j,1}^iz_1&&=\int \frac{\partial u}{\partial x_1}R_k^i\cdot \nabla w(x-R_k^iy_j)dx
+O(e^{-\frac{1+\xi}{2}\ell})\\
&&=\int_{B_{\frac{\ell}{2}}(R_k^iy_j)}\frac{\partial u}{\partial x_1}R_k^i\cdot \nabla w(x-R_k^iy_j)dx\\
&&+\int_{\R^N-B_{\frac{\ell}{2}}(R_k^iy_j)}\frac{\partial u}{\partial x_1}R_k^i\cdot \nabla w(x-R_k^iy_j)dx+O(e^{-\frac{1+\xi}{2}\ell})\\
&&=\cos\theta_i\int (\frac{\partial w}{\partial x_1})^2dx+O(e^{-\frac{(1+\xi)\ell}{2}}).
\end{eqnarray*}
Similarly, one can get that
\begin{eqnarray*}
\int Z_{j,2}^iz_1&&=\int \frac{\partial u}{\partial x_1}R_k^{i,\perp}\cdot \nabla w(x-R_k^iy_j)dx+O(e^{-\frac{1+\xi}{2}\ell})\\
&&=\sin\theta_i\int (\frac{\partial w}{\partial x_1})^2dx+O(e^{-\frac{(1+\xi)\ell}{2}}),
\end{eqnarray*}
and for $\alpha=3,\cdots,N$
\begin{eqnarray*}
\int Z_{j,\alpha}^iz_1=\int \frac{\partial u}{\partial x_1}\frac{\partial}{\partial x_\alpha}w(x-R_k^iy_j)=0
\end{eqnarray*}
by the evenness of $u$ in $x_\alpha$.

Moreover, we have for $j=m+2,\cdots,m+2n$,
\begin{equation*}
\int Z_{j,1}^iz_1=-\sin(\theta_i+\frac{\pi}{k})\int(\frac{\partial w}{\partial x_1})^2+O(e^{-\frac{(1+\xi)\ell}{2}}),
\end{equation*}
\begin{equation*}
\int Z_{j,2}^iz_1=\cos(\theta_i+\frac{\pi}{k})\int(\frac{\partial w}{\partial x_1})^2+O(e^{-\frac{(1+\xi)\ell}{2}}),
\end{equation*}
and
\begin{equation*}
\int Z_{j,\alpha}^iz_1=0
\end{equation*}
for $\alpha=3,\cdots,N$.

Direct consequence of the above calculation is that
\begin{eqnarray*}
\sum_{\alpha=1}^N{\bf c}_\alpha\cdot \int {\bf Z}_\alpha z_1&&={\bf c_1}\cdot \left(\begin{array}{c}
{\bf cos}_k\\
{\bf cos}_k\\
{\bf cos(\theta_0)}\\
\vdots\\
{\bf cos(\theta_{k-1})}\\
-{\bf sin(\theta_0+\frac{\pi}{k}) }\\
\vdots\\
-{\bf sin(\theta_0+\frac{\pi}{k})}
\end{array}
\right)+{\bf c}_2\cdot \left(\begin{array}{c}
{\bf sin}_k\\
{\bf sin}_k\\
{\bf sin(\theta_0)}\\
\vdots\\
{\bf sin(\theta_{k-1})}\\
{\bf cos(\theta_0+\frac{\pi}{k}) }\\
\vdots\\
{\bf cos(\theta_0+\frac{\pi}{k})}
\end{array}
\right)\\
&&+O(e^{-\frac{(1+\xi)\ell}{2}})\mathcal{L}\left(\begin{array}{c}{\bf c}_1\\ \vdots\\ {\bf c}_N
\end{array}\right)
\end{eqnarray*}
where $\mathcal{L}$ is linear function, whose coefficients are uniformly bounded in $\ell$ as $\ell$ to $\infty$. Thus (\ref{eq301}) follows easily. The proof of (\ref{eq302}) to (\ref{eq305}) are similar and we leave it to the reader.
\end{proof}

\medskip
\noindent
{\bf Fact 2:} Since $L(\tilde{\varphi})=0$, we have that
\begin{equation}\label{eq309}
L(\varphi^\perp)=-\sum_{\alpha=1}^N {\bf c}_\alpha \cdot L({\bf Z}_\alpha).
\end{equation}

We first have the following a priori estimate on the above equation:
\begin{proposition}\label{pro302}
Assume $h$ is a function such that $\|h\|_*<\infty$ where the norm $\|\cdot\|_*$ is defined in (\ref{norm}). Let $\phi$ be solution of the following equation
\begin{equation}\label{eq306}
L(\phi)=h, \ \int \phi Z_{j,\alpha}^i=0 \mbox{ for } i=0,\cdots,k-1,j=1,\cdots,m+2n,\alpha=1,\cdots,N.
\end{equation}
Then for $\ell$ large, $\phi$ must satisfy the following estimate:
\begin{equation}
\|\phi\|_*\leq C\|h\|_*.
\end{equation}
\end{proposition}
\begin{proof}
This can be proved by contradiction. Assume that there exists $\ell_n\to \infty$, and corresponding $\phi_n, h_n$ to (\ref{eq306}), such that
\begin{equation}
\|\phi_n\|_*=1, \ \|h_n\|_*\to 0 \mbox{ as }n\to \infty.
\end{equation}

In the following, we omit the index $n$ without confusion. Following the argument in Proposition 3.1 in \cite{MPW}, one can get that there exists $R_k^iy_j$ such that
\begin{equation}\label{eq307}
\|\phi\|_{L^\infty(B(R_k^iy_j,\rho))}\geq C,
\end{equation}
for some fixed $C$ and $\rho$ large. Using elliptic estimates together with Ascoli-Arzela theorem, we can find a sequence  of $R_k^iy_j$ and $\phi(x+R_k^iy_j)$ converges to $\phi_\infty$ which is a solutions of
\begin{equation*}
\Delta \phi_\infty-\phi_\infty+pw^{p-1}\phi_\infty=0
\end{equation*}
and satisfies the following orthogonal conditions:
\begin{equation*}
\int \phi_\infty \frac{\partial w}{\partial x_\alpha}=0 , \ \alpha=1,\cdots,N.
\end{equation*}
Thus $\phi_\infty=0$, this is a contradiction with (\ref{eq307}). This completes the proof.
\end{proof}

Since
\begin{equation*}
L(Z_{j,\alpha}^i)=p(|u|^{p-1}-w^{p-1}(x-R_k^iy_j))Z_{j,\alpha}^i+O(e^{-\frac{1+\xi}{2}\ell}),
\end{equation*}
one can easily get that
\begin{equation}\label{eq308}
\|L(Z_{j,\alpha}^i)\|_*\leq Ce^{-\frac{1+\xi}{2}\ell}
\end{equation}
for some $\xi$ independent of $\ell$ large where we have used the estimate (\ref{pialpha}).

So from Proposition \ref{pro302}, and estimate (\ref{eq308}), we get that
\begin{equation}\label{varphi}
\|\varphi^\perp\|_*\leq Ce^{-\frac{1+\xi}{2}\ell}\sum_{\alpha=1}^N\|{\bf c}_\alpha\|.
\end{equation}

\medskip
\noindent
{\bf Fact 3:}
Let us now multiply (\ref{eq309}) against $Z_{j,\alpha}^i$ for $i=0,\cdots,k-1$, $j=1,\cdots, m+2n$ and $\alpha=1,\cdots,N$. After integrating in $\R^N$, we get a linear system of $(2n+m)\times k$ equations in the $(2n+m)\times k$ coefficients ${\bf c}$ of the form
\begin{equation}\label{eq313}
M\left(\begin{array}{c}
{\bf c}_1\\
\vdots\\
{\bf c}_N
\end{array}\right)=-\left(\begin{array}{c}
{\bf r}_1\\
\vdots\\
{\bf r}_N
\end{array}
\right) \mbox{ with } {\bf r}_\alpha=\left(\begin{array}{c}
\int L(\varphi^\perp)Z_{\bf{v},\alpha}\\
\int L(\varphi^\perp) Z_{Y_1,\alpha}^0\\
\vdots\\
\int L(\varphi^\perp)Z_{Y_1,\alpha}^{k-1}\\
\int L(\varphi^\perp)Z_{Y_2,\alpha}^0\\
\vdots\\
\int L(\varphi^\perp)Z_{Y_2,\alpha}^{k-1}
\end{array}
\right)
\end{equation}

Next let us analysis the matrix $M$. A first observation is that, if $\alpha$ is any of the indices $\{1,2\}$ and $\beta$ is any of the index in $\{3,\cdots,N\}$,
\begin{equation*}
\int L(Z_{i,\beta}^t)Z_{j,\alpha}^s=0 \mbox{ for \ any } i,j=1,\cdots, 2n+m,\  s,t=0,\cdots,k-1.
\end{equation*}
This fact implies that the matrix $M$ has the form
\begin{equation}\label{eq310}
M=\left(\begin{array}{cc}
M_1&0\\
0&M_2
\end{array}
\right)
\end{equation}
where $M_1$ is a matrix of dimension $(2\times (2n+m))^2$ and $M_2$ is a matrix of dimension $((N-2)\times(2n+m))^2$.

Since
\begin{equation*}
\int L(Z_{i,\alpha}^s)Z_{j,\beta}^t=\int L(Z_{j,\beta}^t)Z_{i,\alpha}^s,
\end{equation*}
we can write that
\begin{equation}\label{eq311}
M_1=\left(\begin{array}{cc}
A&B\\
B^t&C
\end{array}
\right)
\end{equation}
where $A,B,C$ are square matrix of dimension $((2n+m)\times k)^2$, with $A,C$ symmetric. More precisely
\begin{eqnarray*}
&&A=(\int L(Z_{i,1}^s)Z_{j,1}^t)_{i,j=1,\cdots,2n+m,\ s,t=0,\cdots,k-1},\\
&&B=(\int L(Z_{i,1}^s)Z_{j,2}^t)_{i,j=1,\cdots,2n+m,\ s,t=0,\cdots,k-1},\\
&&C=(\int L(Z_{i,2}^s)Z_{j,2}^t)_{i,j=1,\cdots,2n+m,\ s,t=0,\cdots,k-1}.
\end{eqnarray*}

Furthermore, by symmetry again, since
\begin{equation*}
\int L(Z_{i,\alpha}^s)Z_{j,\beta}^t=0, \mbox{ if }\alpha\neq \beta, \alpha,\beta=3,\cdots,N,
\end{equation*}
the matrix $M_2$ has the form
\begin{equation}\label{eq312}
M_2=\left(\begin{array}{ccccc}
H_3&0&0&0&0\\
0&H_4&0&0&0\\
0&\ddots&\ddots&\ddots&0\\
0&\hdots&0&0&H_N
\end{array}
\right)
\end{equation}
where $H_\alpha$ are square matrices of dimension $(2n+m)^2$, and each of them is symmetric. The matrix $H_\alpha$ are defined by
\begin{equation}
H_\alpha=(\int L(Z_{i,\alpha}^s)Z_{j,\alpha}^t)_{i,j=1,\cdots,2n+m} \mbox{ for }\alpha=3,\cdots,N.
\end{equation}
Thus, given the form of the matrix $M$ as described in (\ref{eq310}), (\ref{eq311}) , and (\ref{eq312}),  system (\ref{eq313}) is equivalent to
\begin{equation*}
M_1\left(\begin{array}{c}
{\bf c}_1\\
{\bf c}_2
\end{array}
\right)=-\left(\begin{array}{c}
{\bf r}_1\\
{\bf r}_2
\end{array}
\right), \ H_\alpha{\bf c}_\alpha=-{\bf r}_\alpha, \mbox{ for }\alpha=3,\cdots,N.
\end{equation*}
where the vectors ${\bf r}_\alpha$ are defined in (\ref{eq313}).

We will show that the above system is solvable. Indeed we have the validity of the following:
\begin{proposition}\label{pro303}
There exists $\ell_0>0$ and $C$, such that for $\ell>\ell_0$, system (\ref{eq313}) is solvable. Furthermore, the solution has the form
\begin{eqnarray*}
&\left(\begin{array}{c}
{\bf c}_1\\
{\bf c}_2
\end{array}
\right)=\left(\begin{array}{c}
{\bf v}_1\\
{\bf v}_2
\end{array}
\right)\\
&+s_1\left(\begin{array}{c}
{\bf cos}_k\\
{\bf cos}_k\\
{\bf cos(\theta_0)}\\
\vdots\\
{\bf cos(\theta_{k-1})}\\
-{\bf sin(\theta_0+\frac{\pi}{k}) }\\
\vdots\\
-{\bf sin(\theta_0+\frac{\pi}{k})}\\
{\bf sin}_k\\
{\bf sin}_k\\
{\bf sin(\theta_0)}\\
\vdots\\
{\bf sin(\theta_{k-1})}\\
{\bf cos(\theta_0+\frac{\pi}{k}) }\\
\vdots\\
{\bf cos(\theta_0+\frac{\pi}{k})}
\end{array}
\right)+s_2\left(\begin{array}{c}
{\bf sin}_k\\
{\bf sin}_k\\
{\bf sin(\theta_0)}\\
\vdots\\
{\bf sin(\theta_{k-1})}\\
{\bf cos(\theta_0+\frac{\pi}{k}) }\\
\vdots\\
{\bf cos(\theta_0+\frac{\pi}{k})}\\
-{\bf cos}_k\\
-{\bf cos}_k\\
-{\bf cos(\theta_0)}\\
\vdots\\
-{\bf cos(\theta_{k-1})}\\
{\bf sin(\theta_0+\frac{\pi}{k}) }\\
\vdots\\
{\bf sin(\theta_0+\frac{\pi}{k})}
\end{array}
\right)+s_3\left(\begin{array}{c}
{\bf 0}_k\\
{\bf 0}_k\\
{\bf R_k^0y_j\cdot R_k^{0,\perp}}\\
\vdots\\
{\bf R_k^{k-1}y_j\cdot R_k^{k-1,\perp}}\\
{\bf R_k^0z_h\cdot n_0}\\
\vdots\\
{\bf R_k^{k-1}z_h\cdot n_{k-1}}\\
-{\bf |y_1|}_k\\
-{\bf |y_{m+1}|}_k\\
-{\bf R_k^0y_j\cdot R_k^0}\\
\vdots\\
-{\bf R_k^{k-1}y_j\cdot R_k^{k-1}}\\
-{\bf R_k^0z_h\cdot t_0}\\
\vdots\\
-{\bf R_k^{k-1}z_h\cdot t_{k-1}}
\end{array}
\right)\\
&:=\left(\begin{array}{c}
{\bf v}_1\\
{\bf v}_2
\end{array}
\right)+s_1{\bf w}_1+s_2{\bf w}_1+s_3{\bf w}_3
\end{eqnarray*}
and
\begin{eqnarray*}
{\bf c}_\alpha&&={\bf v}_\alpha+s_{\alpha1}\left(\begin{array}{c}
{\bf 1}_k\\
{\bf 1}_k\\
{\bf 1}_{m-1}\\
\vdots\\
{\bf 1}_{m-1}\\
{\bf 1}_{2n-1}\\
\vdots\\
{\bf 1}_{2n-1}
\end{array}
\right)+s_{\alpha2}\left(\begin{array}{c}
|y_1|{\bf cos}_k\\
|y_{m+1}|{\bf cos}_k\\
{\bf R_k^0y_j\cdot e_1}\\
\vdots\\
{\bf R_k^{k-1}y_j\cdot e_1}\\
{\bf R_k^0z_h\cdot e_1}\\
\vdots\\
{\bf R_k^{k-1}z_h\cdot e_1}
\end{array}
\right)+s_{\alpha3}\left(\begin{array}{c}
|y_1|{\bf sin}_k\\
|y_{m+1|}{\bf sin}_k\\
{\bf R_k^0y_j\cdot e_2}\\
\vdots\\
{\bf R_k^{k-1}y_j\cdot e_2}\\
{\bf R_k^0z_h\cdot e_2}\\
\vdots\\
{\bf R_k^{k-1}z_h\cdot e_2}
\end{array}
\right)\\
&&:={\bf v}_\alpha+s_{\alpha1}{\bf w}_4+s_{\alpha 2}{\bf w}_5+s_{\alpha 3}{\bf w}_6
\end{eqnarray*}
for any $s_1,s_2,s_3, s_{\alpha1},s_{\alpha2},s_{\alpha3}\in \R$, where the vectors ${\bf v}_\alpha$ are fixed vectors with
\begin{equation}
\|{\bf v}_\alpha\|\leq C\ell^\tau e^{\frac{1-\xi}{2}\ell}\|\varphi^\perp\|_*
\end{equation}
for some $\tau,\xi>0$.
\end{proposition}

Conditions (\ref{eq301})-(\ref{eq305}) guarantee that the solutions ${\bf c}_\alpha$ to (\ref{eq313}) is indeed unique. Furthermore, we will show that there exists a constant $C$ such that
\begin{equation*}
\sum_{\alpha=1}^N\|{\bf c}_\alpha\|\leq C\ell^\tau e^{\frac{1-\xi}{2}\ell}\|\varphi^\perp\|,
\end{equation*}
and
\begin{equation}
\|\varphi^\perp\|_*\leq Ce^{-\frac{1+\xi}{2}\ell}\sum_{\alpha=1}^N\|{\bf c}_\alpha\|.
\end{equation}
Combining the above two estimates,we have that
\begin{equation*}
{\bf c}_\alpha=0 \mbox{ for }\alpha=1,\cdots,N \mbox{ and }\varphi^\perp=0.
\end{equation*}

\setcounter{equation}{0}
\section{Analysis of the matrix $M$}\label{sec4}
This section is devoted to the analysis of kernels and eigenvalues of the matrices $A,B,C,H_\alpha$. The main result in this section is the following solvability condition of the matrix $M$:

\begin{proposition}\label{pro401}
Part a.

There exists $\ell_0>0$ such that for $\ell>\ell_0$, System
\begin{equation*}
M_1\left(\begin{array}{c}
{\bf c}_1\\
{\bf c}_2\
\end{array}
\right)=
\left(\begin{array}{c}
{\bf r}_1\\
{\bf r}_2\
\end{array}
\right)
\end{equation*}
is solvable if
\begin{equation*}
\left(\begin{array}{c}
{\bf r}_1\\
{\bf r}_2\
\end{array}
\right)\cdot {\bf w}_1=\left(\begin{array}{c}{\bf r}_1\\
{\bf r}_2\
\end{array}
\right)\cdot {\bf w}_2=\left(\begin{array}{c}
{\bf r}_1\\
{\bf r}_2\
\end{array}
\right)\cdot {\bf w}_3=0.
\end{equation*}
Furthurmore the solution has the form
\begin{equation}
\left(\begin{array}{c}
{\bf c}_1\\
{\bf c}_2\
\end{array}
\right)=\left(\begin{array}{c}
{\bf v}_1\\
{\bf v}_2\
\end{array}
\right)+t_1{\bf w}_1+t_2{\bf w}_2+t_3{\bf w}_3
\end{equation}
for all $t_i\in\R$ and with $\left(\begin{array}{c}
{\bf v}_1\\
{\bf v}_2\
\end{array}
\right)$ a fixed vector such that
\begin{equation}
\|\left(\begin{array}{c}
{\bf v}_1\\
{\bf v}_2\
\end{array}
\right)\|\leq C\ell^\tau e^{\ell}\|\left(\begin{array}{c}
{\bf r}_1\\
{\bf r}_2\
\end{array}
\right)\|.
\end{equation}

Part b.

Let $\alpha=3,\cdots,N$. There exists $\ell_0>0$,for any $\ell>\ell_0$,
\begin{equation}
H_\alpha({\bf c}_\alpha)={\bf r}_\alpha
\end{equation}
is solvable if

\begin{equation*}
{\bf r}_\alpha\cdot {\bf w}_4={\bf r}_\alpha\cdot {\bf w}_5={\bf r}_\alpha\cdot {\bf w}_6=0.
\end{equation*}

Furthermore, the solution has the form
\begin{equation}
{\bf c}_\alpha={\bf v}_\alpha+t_{\alpha 1}w_4+t_{\alpha 2}w_5+t_{\alpha 3}w_6,
\end{equation}
for all $t_{\alpha i}\in\R$ and ${\bf v}_\alpha$ a fixed vector such that
\begin{equation}
\|{\bf v}_\alpha\|\leq C\ell^\tau e^{\ell}\|{\bf r}_\alpha\|.
\end{equation}

\end{proposition}

Before we prove the above Proposition, we first need to introduce some notations.

For all $\bar{n}\geq 2$, we define the $\bar{n}\times \bar{n}$ matrix
\begin{equation}
T_{\bar{n}}=\left(\begin{array}{ccccc}
2&-1&0&\cdots&0\\
-1&2&\ddots&\ddots&\vdots\\
0&\ddots&\ddots&\ddots&0\\
\vdots&\ddots&\ddots&2&-1\\
0&\cdots&0&-1&2
\end{array}
\right)
\end{equation}
In application, the integer $\bar{n}$ will be equal to $m-1$ or $2n-1$.

It is easy to check that the inverse of $T_{\bar{n}}$ is the matrix whose entries are given by
\begin{equation}
(T_{\bar{n}}^{-1})_{ij}=\min(i,j)-\frac{ij}{\bar{n}+1}.
\end{equation}
We define the vectors $S^\downarrow$ and $S^\uparrow$ by
\begin{equation}
 T_{\bar{n}} \, S^{\downarrow} : =
\begin{pmatrix}
0  \\
 \vdots \\
0 \\
1 \\
\end{pmatrix}
\in\R^{\bar{n}}
\qquad
T_{\bar{n}} \, S^{\uparrow}: =
\begin{pmatrix}
1 \\
0 \\
\vdots \\
0 \\
\end{pmatrix}
\in \R^{\bar{n}} \, .
\end{equation}
It is immediate to check that
\begin{equation}
\label{defRexplicit}
 S^{\uparrow} : =
\begin{pmatrix} {\bar{n}\over \bar{n}+1}  \\
{\bar{n}-1 \over \bar{n}+1}\\
 \vdots \\
{2\over \bar{n}+1} \\
{1 \over \bar{n}+1} \\
\end{pmatrix}
\in\R^{\bar{n}}
\qquad
S^{\downarrow}: =
\begin{pmatrix}
{1 \over \bar{n}+1} \\
{2\over \bar{n}+1} \\
\vdots \\
{\bar{n}-1 \over \bar{n}+1} \\
{\bar{n}\over \bar{n}+1}
\end{pmatrix}
\in \R^{\bar{n}} \, .
\end{equation}
\medskip

We also introduce the following vectors:
\begin{equation*}
{\bf d}_{L,i}=(c,0\cdots,0)\in \R^i
\end{equation*}
\begin{equation}
{\bf d}_{R,i}=(0,\cdots,0,c)\in \R^{i}
\end{equation}
and
\begin{equation}
{\bf d}_{i}=(d,d,\cdots,d)\in \R^{i}.
\end{equation}
In application $i=m-1$ or $2n-1$.

As we will see below that the circulant matrix will play important role in our proof. We recall the definition of circulant matrix.

A circulant matrix $X$ of dimension $k\times k$ has the form
\begin{equation}
X=\left(\begin{array}{ccccc}
x_0 & x_1&\cdots & x_{k-2} & x_{k-1}\\
x_{k-1} & x_0 & x_1&\cdots & x_{k-2}\\
\cdots & x_{k-1} & x_{0} & x_1&\cdots\\
\cdots&\cdots&\cdots&\cdots&\cdots\\
\cdots&\cdots&\cdots&\cdots & x_1\\
x_1&\cdots&\cdots& x_{k-1} & x_0
\end{array}
\right)
\end{equation}
or equivalently, if $x_{ij}, i,j=1,\cdots,k$ are the entrances of the matrix $X$, then
\begin{equation*}
x_{ij}=x_{1,|i-j|+1}.
\end{equation*}
In particular, in order to know a circulant matrix, it is enough to know the entrance of the first row. We denote by
\begin{equation}
X=\operatorname{Cir}\{(x_0,x_1,\cdots,x_{k-1})\}
\end{equation}
the above mentioned circulant matrix.

The eigenvalue of a circulant matrix $X$ are given by the explicit formula
\begin{equation}\label{eigenvalue}
\eta_s=\sum_{l=1}^{k-1}x_l e^{\frac{2\pi s}{k}il}, \ s=0,\cdots,k-1
\end{equation}
and with corresponding normalized eigenvectors defined by
\begin{equation}
E_s=k^{-\frac{1}{2}}\left(\begin{array}{c}
1\\ e^{\frac{2\pi s}{k}i}\\ e^{\frac{2\pi s}{k}i2}\\ \vdots\\ e^{\frac{2\pi s}{k}i(k-1)}
\end{array}\right).
\end{equation}
Observe that any circulant matrix $X$ can be diagonalized
\begin{equation*}
X=P D_{X}P^t
\end{equation*}
where $D_X$ is the diagonal matrix
\begin{equation*}
D_X=\operatorname{diag}(\eta_0,\eta_1,\cdots,\eta_{k-1})
\end{equation*}
and $P$ is the $k\times k$ matrix defined by
\begin{equation}\label{P}
P=(E_0|E_1|\cdots|E_{k-1}).
\end{equation}

\noindent

Next we will analyze the matrix $H_\alpha$ and $M_1$.

\medskip
\noindent
{\bf The computation of $H_\alpha$.}
We first analyze the kernels of the matrix $H_\alpha$.

First we denote by
\begin{equation*}
\frac{\Psi_2(\bar{\ell})}{\Psi_2(\ell)}=\frac{\delta_2}{2\sin\frac{\pi}{k}}.
\end{equation*}

Dividing both sides of the equation $H_\alpha({\bf c}_\alpha)=0$ by $\Psi_2(\ell)$, we get that
\begin{equation*}
\bar{H}_\alpha({\bf c}_\alpha)=0
\end{equation*}
where $\bar{H}_\alpha=\frac{H_\alpha}{\Psi_2(\ell)}$.

By the computations in Section \ref{sec7}, we know that $\bar{H}_\alpha$ has the form
\begin{equation}\label{eq401}
\bar{H}_\alpha=\left(\begin{array}{cccc}
H_{\alpha,1}&0&H_{\alpha,2}&0\\
0&H_{\alpha,3}&H_{\alpha,4}&H_{\alpha,5}\\
H_{\alpha,2}^t&H_{\alpha,4}^t&H_{\alpha,6}&0\\
0&H_{\alpha,5}^t&0&H_{\alpha,7}
\end{array}
\right)+O(e^{-\xi\ell})
\end{equation}
where
\begin{equation*}
H_{\alpha,1}=\left(\begin{array}{cccccc}
-1-\frac{\delta_2}{\sin\frac{\pi}{k}}&\frac{\delta_2}{2\sin\frac{\pi}{k}}&0&\cdots&0&\frac{\delta_2}{2\sin\frac{\pi}{k}}\\
\frac{\delta_2}{2\sin\frac{\pi}{k}}&-1-\frac{\delta_2}{\sin\frac{\pi}{k}}&\frac{\delta_2}{2\sin\frac{\pi}{k}}&0&\cdots&0\\
0&\ddots&\ddots&\ddots&\ddots&0\\
\vdots&\ddots&\ddots&\ddots&\ddots&\vdots\\
0&\cdots&0&\frac{\delta_2}{2\sin\frac{\pi}{k}}&-1-\frac{\delta_2}{\sin\frac{\pi}{k}}&\frac{\delta_2}{2\sin\frac{\pi}{k}}\\
\frac{\delta_2}{2\sin\frac{\pi}{k}}&0&\cdots&0&\frac{\delta_2}{2\sin\frac{\pi}{k}}&-1-\frac{\delta_2}{\sin\frac{\pi}{k}}
\end{array}
\right)_{k\times k}
\end{equation*}
and
\begin{equation*}
H_{\alpha,2}=\left(\begin{array}{cccc}
{\bf 1}_{L,m-1}&{\bf 0}_{m-1}&\cdots&{\bf 0}_{m-1}\\
{\bf 0}_{m-1}&{\bf 1}_{L,m-1}&\cdots&{\bf 0}_{m-1}\\
\cdots&\cdots&\cdots&\cdots\\
{\bf 0}_{m-1}&\cdots&\cdots&{\bf 1}_{L,m-1}
\end{array}
\right)_{[(m-1)\times k]\times k},
\end{equation*}

\begin{equation*}
H_{\alpha,3}=\left(\begin{array}{cccc}
\frac{\delta_2}{\sin\frac{\pi}{k}}-1&0&\cdots&0\\
0&\frac{\delta_2}{\sin\frac{\pi}{k}}-1&0&0\\
\vdots&\ddots&\ddots&0\\
0&\cdots&\cdots&\frac{\delta_2}{\sin\frac{\pi}{k}}-1
\end{array}
\right)_{k\times k},
\end{equation*}

\begin{equation*}
H_{\alpha,4}=\left(\begin{array}{cccc}
{\bf 1}_{R,m-1}&{\bf 0}_{m-1}&\cdots&{\bf 0}_{m-1}\\
{\bf 0}_{m-1}&{\bf 1}_{R,m-1}&\cdots&{\bf 0}_{m-1}\\
\cdots&\cdots&\cdots&\cdots\\
{\bf 0}_{m-1}&\cdots&\cdots&{\bf 1}_{R,m-1}\end{array}
\right)_{[(m-1)\times k]\times k},
\end{equation*}

\begin{equation*}
H_{\alpha,5}=\left(\begin{array}{cccc}
-{\bf \frac{\delta_2}{2\sin\frac{\pi}{k}}}_{L,2n-1}&{\bf 0}_{2n-1}&\cdots&-{\bf \frac{\delta_2}{2\sin\frac{\pi}{k}}}_{R,2n-1}\\
-{\bf \frac{\delta_2}{2\sin\frac{\pi}{k}}}_{R,2n-1}&-{\bf \frac{\delta_2}{2\sin\frac{\pi}{k}}}_{L,2n-1}&\cdots&{\bf 0}_{2n-1}\\
\cdots&\cdots&\cdots&\cdots\\
{\bf 0}_{2n-1}&\cdots&-{\bf \frac{\delta_2}{2\sin\frac{\pi}{k}}}_{R,2n-1}&-{\bf \frac{\delta_2}{2\sin\frac{\pi}{k}}}_{L,2n-1}
\end{array}
\right)_{[(2n-1)\times k]\times k},
\end{equation*}
\begin{equation*}
H_{\alpha,6}=\left(\begin{array}{cccc}
-T_{m-1}&0&\cdots&0\\
0&-T_{m-1}&\cdots&0\\
\vdots&\vdots&\vdots&-T_{m-1}
\end{array}
\right),
\end{equation*}
and
\begin{equation*}
H_{\alpha,7}=\left(\begin{array}{cccc}
\frac{\delta_2}{2\sin\frac{\pi}{k}}T_{2n-1}&0&\cdots&0\\
0&\frac{\delta_2}{2\sin\frac{\pi}{k}}T_{2n-1}&\cdots&0\\
\vdots&\vdots&\vdots&\frac{\delta_2}{2\sin\frac{\pi}{k}}T_{2n-1}
\end{array}
\right).
\end{equation*}
\medskip

We want to analyze the eigenvalues of the matrix $\bar{H}_\alpha$. If $\bar{H}_\alpha{\bf a}=0$. First by considering the third row of the matrix $\bar{H}_\alpha$ written in the form (\ref{eq401}), one can get that
\begin{equation*}
(-T_{m-1}+O(e^{-\xi}\ell))({\bf a}_{Y_1,i})+\left(\begin{array}{c}
a_{1}^i\\
0\\
\vdots\\
0\\
a_{m+1}^i
\end{array}
\right)=O(e^{-\xi}\ell){\bf a}_{v},
\end{equation*}
and
\begin{equation*}
(T_{2n-1}+O(e^{-\xi}\ell))({\bf a}_{Y_2}^i)-\left(\begin{array}{c}
a_{m+1}^i\\
0\\
\vdots\\
0\\
a_{m+1}^{i+1}
\end{array}
\right)=O(e^{-\xi\ell}){\bf a}_v.
\end{equation*}

From the above two equations, using (\ref{defRexplicit}),  one has that for $i=0,\cdots,k-1$
\begin{equation}\label{eq403}
\left\{\begin{array}{l}
a_{2}^i=\frac{1}{m}((m-1)a_{1}^i+a_{m+1}^i)+O(e^{-\xi\ell}){\bf a}_v\\
\\
a_{m}^i=\frac{1}{m}(a_{1}^i+(m-1)a_{m+1}^i)+O(e^{-\xi\ell}){\bf a}_v,
\end{array}
\right.
\end{equation}
and
\begin{equation}\label{eq404}
\left\{\begin{array}{l}
a_{m+2}^i=\frac{1}{2n}((2n-1)a_{m+1}^i+a_{m+1}^{i+1})+O(e^{-\xi\ell}){\bf a}_v\\
\\
a_{m+2n}^i=\frac{1}{2n}(a_{m+1}^i+(2n-1)a_{m+1}^{i+1})+O(e^{-\xi\ell}){\bf a}_v.
\end{array}
\right.
\end{equation}
Next we consider the first and second rows of the matrix $\bar{H}_\alpha$ in (\ref{eq401}), we can get that
\begin{equation*}
H_{\alpha,1}\left(\begin{array}{c}
a_{1}^0\\ a_{1}^1\\
\vdots\\ a_{1}^{k-1}
\end{array}
\right)+\left(\begin{array}{c}
a_{2}^0\\ a_{2}^1\\ \vdots\\ a_{2}^{k-1}\end{array}\right)=O(e^{-\xi\ell}){\bf a}_v
\end{equation*}
and
\begin{equation*}
H_{\alpha,3}\left(\begin{array}{c}
a_{m+1}^0\\ a_{m+1}^1\\
\vdots\\ a_{m+1}^{k-1}
\end{array}
\right)+\left(\begin{array}{c}
a_{m}^0-\frac{\delta_2}{2\sin\frac{\pi}{k}}(a_{m+2}^0+a_{m+2n}^{k-1})\\ a_{m}^1-\frac{\delta_2}{2\sin\frac{\pi}{k}}(a_{m+2}^1-a_{m+2n}^0)\\ \vdots\\ a_{m}^{k-1}-\frac{\delta_2}{2\sin\frac{\pi}{k}}(a_{m+2}^{k-1}-a_{m+2n}^{k-2})\end{array}\right)=O(e^{-\xi\ell}){\bf a}_v.
\end{equation*}
Using the above two equations (\ref{eq403}) and (\ref{eq404}) for $a_{2}^i,a_{m}^i,a_{m+2}^i,a_{m+2n}^i$, the above two equations are reduced to a $2k$ system of $2k$ unknowns $a_{1}^i, a_{m+1}^i$ for $i=0,\cdots,k-1$:
\begin{equation}
\tilde{H}_\alpha\left(\begin{array}{c}
a_{1}^0 \\ \vdots \\ a_{1^{k-1}} \\ a_{m+1}^0\\ \vdots\\ a_{m+1}^{k-1}
\end{array}
\right)=O(e^{-\xi
\ell}){\bf a}_v
\end{equation}
where
\begin{equation*}
\tilde{H}_\alpha=\left(\begin{array}{cc}
\tilde{H}_{\alpha,1}&\tilde{H}_{\alpha,2}\\
\tilde{H}_{\alpha,2}^t&\tilde{H}_{\alpha,3}
\end{array}
\right)
\end{equation*}
and $\tilde{H}_{\alpha,i}$ are all circulant matrices with
\begin{equation}
\tilde{H}_{\alpha,1}=\operatorname{Cir}\{(-\frac{1}{m}-\frac{c_2}{\sin\frac{\pi}{k}},\frac{c_2}{2\sin\frac{\pi}{k}},0,\cdots,0,\frac{c_2}{2\sin\frac{\pi}{k}})\},
\end{equation}
\begin{equation}
\tilde{H}_{\alpha,2}=\operatorname{Cir}\{(\frac{1}{m},0\cdots,0)\},
\end{equation}
\begin{equation}
\tilde{H}_{\alpha,3}=\operatorname{Cir}\{(-\frac{1}{m}+\frac{c_2}{2n\sin\frac{\pi}{k}},-\frac{c_2}{4n\sin\frac{\pi}{k}},0,\cdots,0,-\frac{c_2}{4n\sin\frac{\pi}{k}})\}.
\end{equation}

\medskip
\noindent
{\bf Eigenvalue for $\tilde{H}_{\alpha,1}$}: A direct application of (\ref{eigenvalue}) gives that the eigenvalues of the matrix $\tilde{H}_{\alpha,1}$ are given by
\begin{equation}
h_{1,i}=\frac{\delta_2}{\sin\frac{\pi}{k}}(\cos\frac{2\pi i}{k}-1)-\frac{1}{m}
\end{equation}
for $i=0,\cdots,k-1$.

\noindent
{\bf Eigenvalue for $\tilde{H}_{\alpha,3}$}: The eigenvalues of the matrix $\tilde{H}_{\alpha,3}$ are given by
\begin{equation}
h_{3,i}=\frac{\delta_2}{4n\sin\frac{\pi}{k}}(2-2\cos\frac{2\pi i}{k})-\frac{1}{m}.
\end{equation}
for $i=0,\cdots,k-1$.

Define
\begin{equation}
\mathcal{P}=\left(\begin{array}{cc}
P&0\\
0&P
\end{array}
\right)
\end{equation}
a simple algebra gives that
\begin{equation*}
\tilde{H}_\alpha=\mathcal{P}\left(\begin{array}{cc}
D_1&D_2\\
D_2&D_3
\end{array}
\right)\mathcal{P}^t
\end{equation*}
where
\begin{equation*}
D_1=\operatorname{diag}(h_{1,0},\cdots,h_{1,k-1}),
\end{equation*}
\begin{equation*}
 D_2=\operatorname{diag}(\frac{1}{m},\cdots,\frac{1}{m}),
\end{equation*}
\begin{equation*}
D_3=\operatorname{diag}(h_{3,0},\cdots,h_{3,k-1}).
\end{equation*}

We consider the matrix
\begin{equation}
\mathcal{D}_i=\left(\begin{array}{cc}
h_{1,i}&\frac{1}{m}\\
\frac{1}{m}&h_{3,i}
\end{array}\right).
\end{equation}

The determinant of $\mathcal{D}_i$ is given by
\begin{equation}
\operatorname{Det}(\mathcal{D}_i)=\frac{2\delta_2}{n}\sin^2\frac{\pi i}{k}(1-\frac{\sin^2\frac{\pi i}{k}}{\sin^2\frac{\pi}{k}})(1+O(\frac{1}{\ell}))
\end{equation}
One can check that $\operatorname{Det}(\mathcal{D}_i)=0$, for $i=0,1,k-1$ and $|\operatorname{Det}(\mathcal{D}_i)|\geq \frac{c}{n}$ for $i\geq 2$.

From the above analysis, one can see that the matrix $\bar{H}_\alpha$ has at most three kernels and except zero eigenvalue, the other eigenvalues will have a lower bound $\frac{C}{\ell}$. Moreover, one can directly check that ${\bf w}_4,{\bf w}_5,{\bf w}_6$ are in the kernels of $\bar{H}_\alpha$, so one can get part (a) of proposition \ref{pro401}.

\medskip
\noindent
{\bf The computation of the matrix $M_1$}.
First we denote by
\begin{equation*}
\bar{M}_1=\frac{1}{\Psi_1(\ell)}M_1
\end{equation*}
and introduce the following notations:
\begin{equation}
\Psi_2(\ell)=\frac{\sigma_1}{\ell}\Psi_1(\ell), \ \Psi_2(\bar{\ell})=\frac{\sigma_2}{\ell}\Psi_1(\bar{\ell})
\end{equation}
and
\begin{equation}
\frac{\Psi_1(\bar{\ell})}{\Psi_1(\ell)}=\frac{\sigma_3}{2\sin\frac{\pi}{k}}.
\end{equation}
In fact from (\ref{eq701}), one can get that $\sigma_3=(\frac{d\bar{\ell}}{d\ell})^{-1}$.
By the computation in Section \ref{sec7}, we know that $\bar{M}_1$ can be written in the following form:
\begin{equation}\label{M1}
\bar{M}_1=\left(\begin{array}{cccccccc}
A_{11}^1&0&A_{12}^1&0&B_{11}^1&0&0&0\\
0&A_{11}^3&A_{12}^2&A_{12}^3&0&0&0&B_{12}^1\\
A_{12}^{1,t}&A_{12}^{2,t}&A_{13}^1&0&0&0&0&0\\
0&A_{12}^{3,t}&0&A_{13}^2&0&B_{21}^1&0&0\\
B_{11}^{1,t}&0&0&0&C_{11}^1&0&C_{12}^1&0\\
0&0&0&B_{21}^{1,t}&0&C_{11}^3&C_{12}^2&C_{12}^3\\
0&0&0&0&C_{12}^{1,t}&C_{12}^{2,t}&C_{13}^1&0\\
0&B_{12}^{1,t}&0&0&0&C_{12}^{3,t}&0&C_{13}^2
\end{array}
\right)
\end{equation}
where $A_{11}^1, A_{11}^3$ iare $k\times k$  circulant matrices:

\begin{equation*}
A_{11}^1=\operatorname{Cir}\{(A_{11,0}^1, A_{11,1}^1,0,\cdots, 0,A_{11,k-1}^1)\},
 \end{equation*}
 \begin{equation*}
 A_{11,0}^1=-1-\frac{\sigma_3}{\sin\frac{\pi}{k}}(\sin^2\frac{\pi}{k}+\frac{\sigma_2}{\ell}\cos^2\frac{\pi}{k}),
 \end{equation*}
 \begin{equation*}
 A_{11,1}^1=A_{11,k-1}^1=\frac{\sigma_3}{2\sin\frac{\pi}{k}}(-\sin^2\frac{\pi}{k}+\frac{\sigma_2}{\ell}\cos^2\frac{\pi}{k}),
 \end{equation*}
and
\begin{equation*}
A_{11}^3=\operatorname{Cir}\{(-1+\frac{\sigma_3}{\sin\frac{\pi}{k}}(\sin^2\frac{\pi}{k}+\frac{\sigma_2}{\ell}\cos^2\frac{\pi}{k}),0,\cdots,0)\},
\end{equation*}

and
\begin{equation*}
A_{12}^1=\left(\begin{array}{cccc}
{\bf 1}_{L,m-1}&{\bf 0}_{m-1}&\cdots&{\bf 0}_{m-1}\\
{\bf 0}_{m-1}&{\bf 1}_{L,m-1}&\cdots&{\bf 0}_{m-1}\\
\vdots&\ddots&\ddots&\vdots\\
{\bf 0}_{m-1}&\cdots&\cdots&{\bf 1}_{L,m-1}
\end{array}
\right)_{[(m-1)\times k]\times k},
\end{equation*}
\begin{equation*}
A_{12}^2=\left(\begin{array}{cccc}
{\bf 1}_{R,m-1}&{\bf 0}_{m-1}&\cdots&{\bf 0}_{m-1}\\
{\bf 0}_{m-1}&{\bf 1}_{R,m-1}&\cdots&{\bf 0}_{m-1}\\
\vdots&\ddots&\ddots&\vdots\\
{\bf 0}_{m-1}&\cdots&\cdots&{\bf 1}_{R,m-1}\end{array}
\right)_{[(m-1)\times k]\times k},
\end{equation*}

\begin{equation*}
A_{12}^3=\left(\begin{array}{cccc}
{\bf \frac{\sigma_3}{2}}_{L,2n-1}&{\bf 0}_{2n-1}&\cdots&-{\bf \frac{\sigma_3}{2}}_{R,2n-1}\\
-{\bf \frac{\sigma_3}{2}}_{R,2n-1}&{\bf \frac{\sigma_3}{2}}_{L,2n-1}&\cdots&{\bf 0}_{2n-1}\\
{\bf 0}_{2n-1}&\ddots&\ddots&{\bf 0}_{2n-1}\\
{\bf 0}_{2n-1}&\cdots&-{\bf \frac{\sigma_3}{2}}_{R,2n-1}&{\bf \frac{\sigma_3}{2}}_{L,2n-1}\end{array}
\right)_{[(2n-1)\times k]\times k},
\end{equation*}

and
\begin{equation*}
A_{13}^1=\left(\begin{array}{cccc}
-T_{m-1}&0&\cdots&0\\
0&-T_{m-1}&\cdots&0\\
\vdots&\ddots&\ddots&-T_{m-1}
\end{array}
\right),
\end{equation*}
\begin{equation*}
A_{13}^2=\left(\begin{array}{cccc}
\frac{\sigma_3}{2\sin\frac{\pi}{k}}T_{2n-1}&0&\cdots&0\\
0&\frac{\sigma_3}{2\sin\frac{\pi}{k}}T_{2n-1}&\cdots&0\\
\vdots&\ddots&\ddots&\frac{\sigma_3}{2\sin\frac{\pi}{k}}T_{2n-1}
\end{array}
\right).
\end{equation*}

\medskip

For the matrix $B$, we have $B_{11}^1$ is a $k\times k$ circulant matrix:
\begin{equation*}
B_{11}^1=\operatorname{Cir}\{(0,\frac{\sigma_3\cos\frac{\pi}{k}}{2}(1+\frac{\sigma_2}{\ell}),0,\cdots,\frac{\sigma_3\cos\frac{\pi}{k}}{2}(1+\frac{\sigma_2}{\ell}))\},
\end{equation*}
and
\begin{equation*}
B_{12}^1=\frac{\sigma_2\sigma_3\cos \frac{\pi}{k}}{2\ell \sin\frac{\pi}{k}}\left(\begin{array}{cccc}
-{\bf 1}_{L,2n-1}&{\bf 0}_{2n-1}&\cdots&-{\bf 1}_{R,2n-1}\\
-{\bf 1}_{R,2n-1}&-{\bf 1}_{L,2n-1}&\cdots&{\bf 0}_{2n-1}\\
{\bf 0}_{2n-1}&\ddots&\ddots&{\bf 0}_{2n-1}\\
{\bf 0}_{2n-1}&\cdots&-{\bf 1}_{R,2n-1}&-{\bf 1}_{L,2n-1}\end{array}
\right)_{[(2n-1)\times k]\times k},
\end{equation*}
and
\begin{equation*}
B^{1,t}_{21}=\frac{\sigma_3\cos\frac{\pi}{k}}{2\sin\frac{\pi}{k}}\left(\begin{array}{cccc}
{\bf 1}_{L,2n-1}&{\bf 0}_{2n-1}&\cdots&{\bf 1}_{R,2n-1}\\
{\bf 1}_{R,2n-1}&{\bf 1}_{L,2n-1}&\cdots&{\bf 0}_{2n-1}\\
{\bf 0}_{2n-1}&\ddots&\ddots&{\bf 0}_{2n-1}\\
{\bf 0}_{2n-1}&\cdots&{\bf 1}_{R,2n-1}&{\bf 1}_{L,2n-1}
\end{array}
\right)_{[(2n-1)\times k]\times k}.
\end{equation*}

\medskip

For the matrix $C_{11}^1,C_{11}^3$ are $k\times k$ circulant matrices:
\begin{equation*}
C_{11}^1=\operatorname{Cir}\{(C_{11,0}^1,C_{11,1}^1,0\cdots, 0,C_{11,k-1}^1)\},
\end{equation*}
\begin{equation*}
C_{11,0}^1=-\frac{\sigma_1}{\ell}-\frac{\sigma_3}{\sin\frac{\pi}{k}}(\cos^2\frac{\pi}{k}+\frac{\sigma_2}{\ell}\sin^2\frac{\pi}{k}),
\end{equation*}
\begin{equation*}
C_{11,1}^1=C_{11,k-1}^1= \frac{\sigma_3}{2\sin\frac{\pi}{k}}(\cos^2\frac{\pi}{k}-\frac{\sigma_2}{\ell}\sin^2\frac{\pi}{k}),
\end{equation*}
and
\begin{equation*}
C_{11}^3=\operatorname{Cir}\{(-\frac{\sigma_1}{\ell}+\frac{\sigma_3}{\sin\frac{\pi}{k}}(\cos^2\frac{\pi}{k}+\frac{\sigma_2}{\ell}\sin^2\frac{\pi}{k}),0,\cdots,0)\}.
\end{equation*}

\begin{equation*}
C_{12}^1=\left(\begin{array}{cccc}
{\bf \frac{\sigma_1}{\ell}}_{L,m-1}&{\bf 0}_{m-1}&\cdots&{\bf 0}_{m-1}\\
{\bf 0}_{m-1}&{\bf \frac{\sigma_1}{\ell}}_{L,m-1}&\cdots&{\bf 0}_{m-1}\\
\vdots&\ddots&\ddots&\vdots\\
{\bf 0}_{m-1}&\cdots&\cdots&{\bf \frac{\sigma_1}{\ell}}_{L,m-1}
\end{array}
\right)_{[(m-1)\times k]\times k}
\end{equation*}

\begin{equation*}
C_{12}^2=\left(\begin{array}{cccccccc}
{\bf \frac{\sigma_1}{\ell}}_{R,m-1}&{\bf 0}_{m-1}&\cdots&{\bf 0}_{m-1}\\
{\bf 0}_{m-1}&{\bf \frac{\sigma_1}{\ell}}_{R,m-1}&\cdots&{\bf 0}_{m-1}\\
\vdots&\ddots&\ddots&\vdots\\
{\bf 0}_{m-1}&\cdots&\cdots&{\bf \frac{\sigma_1}{\ell}}_{R,m-1}
\end{array}
\right)_{[(m-1)\times k]\times k},
\end{equation*}
\begin{equation*}
C_{12}^3=\left(\begin{array}{cccc}
{\bf \frac{\sigma_2\sigma_3}{2\ell}}_{L,2n-1}&{\bf 0}_{2n-1}&\cdots&-{\bf \frac{\sigma_2\sigma_3}{2\ell}}_{R,2n-1}\\
-{\bf \frac{\sigma_2\sigma_3}{2\ell}}_{R,2n-1}&{\bf \frac{\sigma_2\sigma_3}{2\ell}}_{L,2n-1}&\cdots&{\bf 0}_{2n-1}\\
{\bf 0}_{2n-1}&\ddots&\ddots&{\bf 0}_{2n-1}\\
{\bf 0}_{2n-1}&\cdots&-{\bf \frac{\sigma_2\sigma_3}{2\ell}}_{R,2n-1}&{\bf \frac{\sigma_2\sigma_3}{2\ell}}_{L,2n-1}
\end{array}
\right)_{[(2n-1)\times k]\times k},
\end{equation*}

and
\begin{equation*}
C_{13}^1=\left(\begin{array}{cccc}
-\frac{\sigma_1}{\ell}T_{m-1}&0&\cdots&0\\
0&-\frac{\sigma_1}{\ell}T_{m-1}&\cdots&0\\
\vdots&\vdots&\vdots&-\frac{\sigma_1}{\ell}T_{m-1}\\
\end{array}
\right),
\end{equation*}
and
\begin{equation*}
C_{13}^2=\left(\begin{array}{cccc}
\frac{\sigma_2\sigma_3}{2\ell\sin\frac{\pi}{k}}T_{2n-1}&0&\cdots&0\\
0&\frac{\sigma_2\sigma_3}{2\ell\sin\frac{\pi}{k}}T_{2n-1}&\cdots&0\\
\vdots&\vdots&\vdots&\frac{\sigma_2\sigma_3}{2\ell\sin\frac{\pi}{k}}T_{2n-1}
\end{array}
\right).
\end{equation*}

\medskip

First by considering the third and fourth rows of $\bar{M}_1{\bf a}=0$ in the form (\ref{M1}), one can get that
\begin{equation}
(-T_{m-1}+O(e^{-\xi\ell})\left(\begin{array}{c}
a_{2,1}^i\\ a_{3,1}^i
\vdots\\
a_{m-1,1}^i\\
a_{m,1}^i
\end{array}
\right)+\left(\begin{array}{c}
a_{1,1}^i\\
0\\
\vdots\\
0\\
a_{m+1,1}^i
\end{array}
\right)=O(e^{-\xi\ell}){\bf a}_v,
\end{equation}
\begin{equation}
T_{2n-1}\left(\begin{array}{c}
a_{m+2,1}^i\\ a_{m+3,1}^i \\ \vdots \\ a_{m+2n-1,1}^i\\ a_{m+2n,1}^i
\end{array}
\right)+\sin\frac{\pi}{k}\left(\begin{array}{c}
a_{m+1,1}^i\\ 0\\ \vdots \\ 0\\ -a_{m+1,1}^{i+1}
\end{array}\right)
+\cos\frac{\pi}{k}\left(\begin{array}{c}
a_{m+1,2}^i\\0\\ \vdots\\
0\\ a_{m+1,2}^i
\end{array}\right)=O(e^{-\xi\ell}){\bf a}_{v}
\end{equation}
From the seventh and eighth rows of matrix $\bar{M}_1$ in (\ref{M1}), we can get that
\begin{equation}
(-T_{m-1}+O(e^{-\xi\ell})\left(\begin{array}{c}
a_{2,2}^i\\
a_{3,2}^i\\
\vdots\\
a_{m-1,2}^i\\
a_{m,2}^i
\end{array}
\right)+\left(\begin{array}{c}
a_{1,2}^i\\
0\\
\vdots\\
0\\
a_{m+1,2}^i
\end{array}
\right)=O(e^{-\xi\ell}){\bf a}_v,
\end{equation}
\begin{equation}
T_{2n-1}\left(\begin{array}{c}
a_{m+2,}^i\\ a_{m+3,2}^i \\ \vdots \\ a_{m+2n-1,2}^i\\ a_{m+2n,2}^i
\end{array}
\right)+\sin\frac{\pi}{k}\left(\begin{array}{c}
a_{m+1,2}^i\\ 0\\ \vdots \\ 0\\ -a_{m+1,2}^{i+1}
\end{array}\right)
-\cos\frac{\pi}{k}\left(\begin{array}{c}
a_{m+1,1}^i\\0\\ \vdots\\
0\\ a_{m+1,1}^i
\end{array}\right)=O(e^{-\xi\ell}){\bf a}_{v}.
\end{equation}
From the above four systems, using (\ref{defRexplicit}), one can solve ${\bf a}_{Y_1,j}^i, {\bf a}_{Y_2,j}^i$ in terms of ${\bf a}_{v}$ for $i=0,\cdots,k-1$, $j=1,2$, in particular we can get that
\begin{equation}\label{eq407}
\left\{\begin{array}{l}
a_{2,1}^i=\frac{1}{m}((m-1)a_{1,1}^i+a_{m+1,1}^i)+O(e^{-\xi\ell}){\bf a}_{v},\\
\\
a_{m,2}^i=\frac{1}{m}(a_{1,1}^i+(m-1)a_{m+1,1}^i)+O(e^{-\xi\ell}){\bf a}_{v},
\end{array}
\right.
\end{equation}
and
\begin{equation}\label{eq408}
\left\{\begin{array}{l}
a_{m+2,1}^i=-\frac{\sin\frac{\pi}{k}}{2n}((2n-1)a_{m+1,1}^i-a_{m+1,1}^{i+1})\\
\ \ \ \ \ \ \ \ \ \ -\frac{\cos\frac{\pi}{k}}{2n}((2n-1)a_{m+1,2}^i+a_{m+1,2}^{i+1})+O(e^{-\xi\ell}){\bf a}_{v}\\
\\
a_{m+2n,1}^i=-\frac{\sin\frac{\pi}{k}}{2n}(a_{m+1,1}^i-(2n-1)a_{m+1,1}^{i+1})\\
\ \ \ \ \ \ \ \ \ \ -\frac{\cos\frac{\pi}{k}}{2n}(a_{m+1,2}^i+(2n-1)a_{m+1,2}^{i+1})+O(e^{-\xi\ell}){\bf a}_{v},
\end{array}
\right.
\end{equation}
and
\begin{equation}\label{eq409}
\left\{\begin{array}{l}
a_{2,2}^i=\frac{1}{m}((m-1)a_{1,2}^i+a_{m+1,2}^i)+O(e^{-\xi\ell}){\bf a}_{v},\\
\\
a_{m,2}^i=\frac{1}{m}(a_{1,2}^i+(m-1)a_{m+1,2}^i)+O(e^{-\xi\ell}){\bf a}_{v},
\end{array}
\right.
\end{equation}
and
\begin{equation}\label{eq410}
\left\{\begin{array}{l}
a_{m+2,2}^i=-\frac{\sin\frac{\pi}{k}}{2n}((2n-1)a_{m+1,2}^i-a_{m+1,2}^{i+1})\\
\ \ \ \ \ \ \ \ \ \ +\frac{\cos\frac{\pi}{k}}{2n}((2n-1)a_{m+1,1}^i+a_{m+1,1}^{i+1})+O(e^{-\xi\ell}){\bf a}_{v}\\
\\
a_{m+2n,2}^i=-\frac{\sin\frac{\pi}{k}}{2n}(a_{m+1,2}^i-(2n-1)a_{m+1,2}^{i+1})\\
\ \ \ \ \ \ \ \ \ \ +\frac{\cos\frac{\pi}{k}}{2n}(a_{m+1,1}^i+(2n-1)a_{m+1,1}^{i+1})+O(e^{-\xi\ell}){\bf a}_{v}.
\end{array}
\right.
\end{equation}
From the first and second rows of the equation in (\ref{M1}), one can get that
\begin{equation}\label{eq411}
A_{11}^1\left(\begin{array}{c}
a_{1,1}^0\\ \vdots \\ a_{1,1}^{k-1}
\end{array}\right)+\left(\begin{array}{c}
a_{2,1}^0\\ \vdots\\ a_{2,1}^{k-1}
\end{array}
\right)+\frac{\sigma_3\cos\frac{\pi}{k}}{2}(1+\frac{\sigma_2}{\ell})\left(\begin{array}{c}
a_{1,2}^1-a_{1,2}^{k-1}\\ a_{1,2}^2-a_{1,2}^0 \\ \vdots \\ a_{1,2}^0-a_{1,2}^{k-2}\end{array}\right)=O(e^{-\xi\ell}){\bf a}_{v},
\end{equation}
\begin{eqnarray}\label{eq412}
A_{11}^3\left(\begin{array}{c}
a_{m+1,1}^0\\ \vdots\\ a_{m+1,1}^{k-1}\end{array}\right)+\left(\begin{array}{c}
a_{m,1}^0\\ \vdots\\ a_{m,1}^{k-1}\end{array}\right)&&+\frac{\sigma_3}{2}\left(\begin{array}{c}
a_{m+2,1}^0-a_{m+2n,1}^{k-1}\\ a_{m+2,1}^1-a_{m+2n,1}^0\\ \vdots\\ a_{m+2,1}^{k-1}-a_{m+2n,1}^{k-2}\end{array}
\right)\\
&&-\frac{\sigma_2\sigma_3\cos\frac{\pi}{k}}{2\ell \sin\frac{\pi}{k}}\left(\begin{array}{c}
a_{m+2,2}^0+a_{m+2n,2}^{k-1}\\ a_{m+2,2}^1+a_{m+2n,2}^0\\ \vdots\\ a_{m+2,2}^{k-1}+a_{m+2n,2}^{k-2}\end{array}
\right)=O(e^{-\xi\ell}){\bf a}_{v}.\nonumber
\end{eqnarray}
From the fifth and sixth rows of (\ref{M1}), one can get that
\begin{equation}\label{eq413}
C_{11}^1\left(\begin{array}{c}
a_{1,2}^0\\ \vdots \\ a_{1,2}^{k-1}
\end{array}\right)+\frac{c_1}{\ell}\left(\begin{array}{c}
a_{2,2}^0\\ \vdots\\ a_{2,2}^{k-1}
\end{array}
\right)-\frac{\sigma_3\cos\frac{\pi}{k}}{2}(1+\frac{c_2}{\ell})\left(\begin{array}{c}
a_{1,1}^1-a_{1,1}^{k-1}\\ a_{1,1}^2-a_{1,1}^0 \\ \vdots \\ a_{1,1}^0-a_{1,1}^{k-2}\end{array}\right)=O(e^{-\xi\ell}){\bf a}_{v},
\end{equation}
\begin{eqnarray}\label{eq414}
C_{11}^3\left(\begin{array}{c}
a_{m+1,2}^0\\ \vdots\\ a_{m+1,2}^{k-1}\end{array}\right)+\frac{\sigma_1}{\ell}\left(\begin{array}{c}
a_{m,2}^0\\ \vdots\\ a_{m,2}^{k-1}\end{array}\right)&&+\frac{\sigma_2\sigma_3}{2\ell}\left(\begin{array}{c}
a_{m+2,2}^0-a_{m+2n,2}^{k-1} \\ a_{m+2,2}^1-a_{m+2n,2}^0\\ \vdots\\ a_{m+2,2}^{k-1}-a_{m+2n,2}^{k-2}\end{array}
\right)\\
&&+\frac{\sigma_3\cos\frac{\pi}{k}}{2\sin\frac{\pi}{k}}\left(\begin{array}{c}
a_{m+2,1}^0+a_{m+2n,1}^{k-1}\\ a_{m+2,1}^1+a_{m+2n,1}^0\\ \vdots\\ a_{m+2,1}^{k-1}+a_{m+2n,1}^{k-2}\end{array}
\right)=O(e^{-\xi\ell}){\bf a}_{v}\nonumber
\end{eqnarray}
Using the equations for $a_{2,j}^i,a_{m,j}^i$ and $a_{m+2,j}^i,a_{m+2n,j}^i$ (\ref{eq407}-(\ref{eq410})), the above system (\ref{eq411})-(\ref{eq414}) can be reduced to $4k$ equations in terms of $4k$ unknowns $a_{1,1}^0,\cdots,a_{1,1}^{k-1}$, $a_{m+1,1}^0,\cdots,a_{m+1,1}^{k-1}$, $a_{1,2}^0,\cdots,a_{1,2}^{k-1}$, $a_{m+1,2}^0,\cdots,a_{m+1,2}^{k-1}$:
\begin{equation}
\left(\begin{array}{cccc}
F_{11}&F_{12}&F_{13}&0\\
F_{12}^t&F_{22}&0&F_{24}\\
F_{13}^t&0&F_{33}&F_{34}\\
0&F_{24}^t&F_{34}^t&F_{44}
\end{array}\right)\left(\begin{array}{c}
a_{1,1}^0\\ \vdots\\ a_{1,1}^{k-1}\\ a_{m+1,1}^0 \\ \vdots\\ a_{m+1,1}^{k-1}\\ a_{1,2}^0\\ \vdots \\ a_{1,2}^{k-1} \\ a_{m+1,2}^0 \\ \vdots \\ a_{m+1,2}^{k-1}
\end{array}
\right)=O(e^{-\xi\ell}){\bf a}_{v}
\end{equation}
where $F_{ij}$ are $k\times k$ circulant matrices given below:

\medskip
\noindent
{\bf The matrix $F_{11}$}. $F_{11}$ is defined by
\begin{equation}
F_{11}=\operatorname{Cir}\{(F_{11,0},F_{11,1},0,\cdots,0,F_{11,k-1})\}
\end{equation}
where
\begin{equation*}
F_{11,0}=-\frac{1}{m}-\frac{\sigma_3}{\sin\frac{\pi}{k}}(\sin^2\frac{\pi}{k}+\frac{\sigma_2}{\ell}\cos^2\frac{\pi}{k}),
\end{equation*}
and
\begin{equation*}
F_{11,1}=F_{11,k-1}=\frac{\sigma_3}{2\sin\frac{\pi}{k}}(-\sin^2\frac{\pi}{k}+\frac{\sigma_2}{\ell}\cos^2\frac{\pi}{k})
\end{equation*}

\noindent
{\bf Eigenvalues of $F_{11}$}. For any $l=0,\cdots,k-1$, the eigenvalues of $F_{11}$ are
\begin{equation*}
f_{11,l}=-\frac{1}{m}-\frac{\sigma_3}{\sin\frac{\pi}{k}}(\sin^2\frac{\pi}{k}+\frac{\sigma_2}{\ell}\cos^2\frac{\pi}{k})+\frac{\sigma_3}{\sin\frac{\pi}{k}}(-\sin^2\frac{\pi}{k}+\frac{\sigma_2}{\ell}\cos^2\frac{\pi}{k})\cos\frac{2l\pi}{k}.
\end{equation*}

\medskip
\noindent
{\bf The matrix $F_{12}$}. The matrix $F_{12}$ is defined by
\begin{equation*}
F_{12}=\operatorname{Cir}\{(\frac{1}{m},0,\cdots,0)\}
\end{equation*}

\noindent
{\bf Eigenvalues of $F_{12}$}. For any $l=0,\cdots,k-1$, the eigenvalues of $F_{12}$ are
\begin{equation*}
f_{12,l}=\frac{1}{m}.
\end{equation*}

\medskip
\noindent
{\bf The matrix $F_{13}$}.  The matrix $F_{13}$ is defined by
\begin{equation*}
F_{13}=\operatorname{Cir}\{(0,F_{13,1},0,\cdots,0,F_{13,k-1})\}
\end{equation*}
where
\begin{equation*}
F_{13,1}=\frac{\sigma_3\cos\frac{\pi}{k}}{2}(1+\frac{\sigma_2}{\ell}), \ F_{13,k-1}=-\frac{\sigma_3\cos\frac{\pi}{k}}{2}(1+\frac{\sigma_2}{\ell}).
\end{equation*}

\noindent
{\bf Eigenvalues of $F_{13}$}. For any $l=0,\cdots,k-1$, the eigenvalues of $F_{12}$ are
\begin{equation}
f_{13,l}=i\sigma_3(1+\frac{\sigma_2}{\ell})\cos\frac{\pi}{k}\sin\frac{2l\pi}{k}
\end{equation}

\medskip
\noindent
{\bf The matrix $F_{22}$}. The matrix $F_{22}$ is defined by
\begin{equation*}
F_{22}=\operatorname{Cir}\{(F_{22,0}, F_{22,1},0,\cdots,0,F_{22,k-1})\}
\end{equation*}
where
\begin{equation*}
F_{22,0}=-\frac{1}{m}+\frac{\sigma_3}{2n\sin\frac{\pi}{k}}(\sin^2\frac{\pi}{k}+\frac{\sigma_2}{\ell}\cos^2\frac{\pi}{k})
\end{equation*}
and
\begin{equation*}
F_{22,1}=F_{22,k-1}=\frac{\sigma_3}{4n\sin\frac{\pi}{k}}(\sin^2\frac{\pi}{k}-\frac{\sigma_2}{\ell}\cos^2\frac{\pi}{k}).
\end{equation*}

\noindent
{\bf Eigenvalue of $F_{22}$}. For any $l=0,\cdots,k-1$, the eigenvalues of $F_{22}$ are

\begin{equation}
f_{22,l}=-\frac{1}{m}+\frac{\sigma_3}{2n\sin\frac{\pi}{k}}(\sin^2\frac{\pi}{k}+\frac{\sigma_2}{\ell}\cos^2\frac{\pi}{k})+\frac{\sigma_3}{2n\sin\frac{\pi}{k}}(\sin^2\frac{\pi}{k}-\frac{\sigma_2}{\ell}\cos^2\frac{\pi}{k})\cos\frac{2l\pi}{k}.
\end{equation}

\medskip
\noindent
{\bf The matrix $F_{24}$}. The matrix $F_{24}$ is defined by
\begin{equation*}
F_{24}=\operatorname{Cir}\{(0, F_{24,1},0,\cdots,0,F_{24,k-1})\}
\end{equation*}
where
\begin{equation*}
F_{24,1}=-\frac{\sigma_3\sin\frac{\pi}{k}}{4n}(1+\frac{\sigma_2}{\ell}), \ F_{24,k-1}=\frac{\sigma_3\sin\frac{\pi}{k}}{4n}(1+\frac{\sigma_2}{\ell}).
\end{equation*}

\noindent
{\bf Eigenvalue of $F_{24}$}. For any $l=0,\cdots,k-1$, the eigenvalues of $F_{24}$ are

\begin{equation*}
f_{24,l}=-\frac{i\sigma_3}{2n}(1+\frac{\sigma_2}{\ell})\cos\frac{\pi}{k}\sin\frac{2l\pi}{k}.
\end{equation*}

\medskip
\noindent
{\bf The matrix $F_{33}$}. The matrix $F_{33}$ is defined by
\begin{equation*}
F_{33}=\operatorname{Cir}\{(F_{33,0},F_{33,1},0,\cdots,0,F_{33,k-1})\}
\end{equation*}
where
\begin{equation*}
F_{33,0}=-\frac{\sigma_1}{m\ell}-\frac{\sigma_3}{\sin\frac{\pi}{k}}(\cos^2\frac{\pi}{k}+\frac{\sigma_2}{\ell}\sin^2\frac{\pi}{k}),
\end{equation*}
and
\begin{equation*}
F_{33,1}=F_{33,k-1}=\frac{\sigma_3}{2\sin\frac{\pi}{k}}(\cos^2\frac{\pi}{k}-\frac{\sigma_2}{\ell}\sin^2\frac{\pi}{k}).
\end{equation*}

\noindent
{\bf Eigenvalue of $F_{33}$}.  For any $l=0,\cdots,k-1$, the eigenvalues of $F_{33}$ are
\begin{equation*}
f_{33,l}=-\frac{\sigma_1}{m\ell}-\frac{\sigma_3}{\sin\frac{\pi}{k}}(\cos^2\frac{\pi}{k}+\frac{\sigma_2}{\ell}\sin^2\frac{\pi}{k})+\frac{\sigma_3}{\sin\frac{\pi}{k}}(\cos^2\frac{\pi}{k}-\frac{\sigma_2}{\ell}\sin^2\frac{\pi}{k})\cos\frac{2l\pi}{k}.
\end{equation*}

\medskip
\noindent
{\bf The matrix $F_{34}$}. The matrix $F_{34}$ is defined by
\begin{equation*}
F_{34}=\operatorname{Cir}\{(\frac{c_1}{m\ell},0,\cdots,0)\}.
\end{equation*}

\noindent
{\bf Eigenvalues of $F_{34}$}. For any $l=0,\cdots,k-1$, the eigenvalues of $F_{34}$ are
\begin{equation*}
f_{34,l}=\frac{\sigma_1}{m\ell}.
\end{equation*}

\medskip
\noindent
{\bf The matrix $F_{44}$}. The matrix $F_{44}$ is defined by
\begin{equation*}
F_{44}=\operatorname{Cir}\{(F_{44,0}, F_{44,1},0,\cdots,0,F_{44,k-1})\}
\end{equation*}
where
\begin{equation*}
F_{44,0}=-\frac{\sigma_1}{m\ell}+\frac{\sigma_3}{2n\sin\frac{\pi}{k}}(\cos^2\frac{\pi}{k}+\frac{\sigma_2}{\ell}\sin^2\frac{\pi}{k}),
\end{equation*}
and
\begin{equation*}
F_{44,1}=F_{44,k-1}=-\frac{\sigma_3}{4n\sin\frac{\pi}{k}}(\cos^2\frac{\pi}{k}-\frac{\sigma_2}{\ell}\sin^2\frac{\pi}{k}).
\end{equation*}

\noindent
{\bf Eigenvalue of $F_{44}$}. For any $l=0,\cdots,k-1$, the eigenvalues of $F_{44}$ are
\begin{equation*}
f_{44,l}=-\frac{\sigma_1}{m\ell}+\frac{\sigma_3}{2n\sin\frac{\pi}{k}}(\cos^2\frac{\pi}{k}+\frac{\sigma_2}{\ell}\sin^2\frac{\pi}{k})-\frac{\sigma_3}{2n\sin\frac{\pi}{k}}(\cos^2\frac{\pi}{k}-\frac{\sigma_2}{\ell}\sin^2\frac{\pi}{k})\cos\frac{2l\pi}{k}.
\end{equation*}

The last part of this section is devoted to the analysis of the matrix
\begin{equation}
F=\left(\begin{array}{cccc}
F_{11}&F_{12}&F_{13}&0\\
F_{12}^t&F_{22}&0&F_{24}\\
F_{13}^t&0&F_{33}&F_{34}\\
0&F_{24}^t&F_{34}^t&F_{44}
\end{array}\right).
\end{equation}

Define
\begin{equation*}
\mathcal{P}_1=\left(\begin{array}{cccc}
P&0&0&0\\
0&P&0&0\\
0&0&P&0\\
0&00&0&P
\end{array}
\right)
\end{equation*}
where $P$ is defined in (\ref{P})
A simple algebra gives that
\begin{equation*}
F=\mathcal{P}_1\left(\begin{array}{cccc}
D_{F_{11}}&D_{F_{12}}&D_{F_{13}}&0\\
D_{F_{12}^t}&D_{F_{22}}&0&D_{F_{24}}\\
D_{F_{13}^t}&0&D_{F_{33}}&D_{F_{34}}\\
0&D_{F_{24}^t}&D_{F_{34}^t}&D_{F_{44}}
\end{array}\right)\mathcal{P}_1^t.
\end{equation*}
Here $D_X$ denotes the diagonal matrix of dimension $k\times k$  whose entrances are given by the eigenvalues of $X$.

Let us now introduce the following matrix
\begin{equation*}
\mathcal{D}_F=\left(\begin{array}{cccc}
\mathcal{D}_{f_0}&0&\cdots&0\\
0&\mathcal{D}_{f_1}&0&\cdots\\
\cdots&\cdots&\cdots&\cdots\\
\cdots&0&0&\mathcal{D}_{f_{k-1}}
\end{array}
\right)
\end{equation*}
where
\begin{equation*}
\mathcal{D}_{f_i}=\left(\begin{array}{cccc}
f_{11,i}&f_{12,i}&f_{13,i}&0\\
f_{12,i}&f_{22,i}&0&f_{24,i}\\
-f_{13,i}&0&f_{33,i}&f_{34,i}\\
0&-f_{24,i}&f_{34,i}&f_{44,i}
\end{array}
\right).
\end{equation*}


By direct calculation, one can check that for $j=0$
\begin{equation*}
\operatorname{Det}(\mathcal{D}_{f_1})=0
\end{equation*}
and $\mathcal{D}_{f_1}$ has only one kernel. The other eigenvalues of $\mathcal{D}_{f_1}$ will satisfy $|\lambda_{1,i}|\geq \frac{C}{\ell^\tau}$ for some constant $C,\tau>0$.

For $j\geq 2$, we have
\begin{eqnarray*}
\operatorname{Det}(\mathcal{D}_{f_j})&&=\frac{n}{2\sigma_3^2(1+\bar{d}_2)^2(1-\bar{a})(1-\bar{b}_j)\bar{b}_j(\bar{a}(1-\bar{b}_j)+\bar{d}_2(1-\bar{a})\bar{d}_j)(\bar{b}_j-\frac{(1-\bar{a})\bar{b}_j\bar{d}_2}{\bar{a}})}\\
&&\times\frac{(\bar{a}-\bar{b}_j)^2\bar{d}_2(\bar{b}_j\bar{d}_2+\bar{a}^2(1+\bar{d}_1)(1+\bar{d}_2)-\bar{a}(1+(2+\bar{d}_1)\bar{d}_2))}{\bar{a}\bar{b}_j(1-\bar{a})(1-\bar{b}_j)(1+\bar{d}_2)^2(\bar{a}+(\bar{a}-1)\bar{d}_2)(\bar{a}(\bar{b}_j-1)+(\bar{a}-1)\bar{b}_j\bar{d}_2)}\\
&&=\frac{(\bar{a}-\bar{b}_j)^2n\bar{d}_2}{2c_3^2\bar{a}^3\bar{b}_j^3(1-\bar{a})(\bar{b}_j-1)^4}(1+O(\frac{1}{\ell})),
\end{eqnarray*}
where
\begin{equation*}
\bar{a}=\sin^2\frac{\pi}{k}, \ \bar{b}_j=\sin^2\frac{j\pi}{k}, \ \bar{d}_1=\frac{c_1}{\ell}, \ \bar{d}_2=\frac{c_2}{\ell}.
\end{equation*}

From the above computation, we know that for $j=1,k-1$, $\operatorname{Det}(\mathcal{D}_{f_j})=0$, and the matrix $\mathcal{D}_{f_j}$ has one kernel and all the other eigenvalues has a lower bounded $\frac{C}{\ell^\tau}$. For $j\neq 0,1,k-1$, the matrix $\mathcal{D}_{f_j}$ is non-degenerate, and the eigenvalues has a lower bounded $\frac{C}{\ell^\tau}$.

From the above analysis, we know that $\bar{M}_1$ has there kernels. And the eigenvalues of $\bar{M}_1$ has a lower bounded $\frac{C}{\ell^\tau}$. Moreover, we can check that ${\bf w}_1,{\bf w}_2,{\bf w}_3$ are in the kernels of $\bar{M}_1$. So we proved part (b).

\setcounter{equation}{0}
\section{Proof of Proposition \ref{pro303}}\label{sec5}
In this section, we prove Proposition \ref{pro303}. A key ingredient in the proof is the estimates on the right hand side of (\ref{eq313}). We have
\begin{proposition}\label{pro501}
There exists positive constants $C$ and $\xi$ such that for $\alpha=1,\cdots,N$,
\begin{equation}\label{eq503}
\|{\bf r}_\alpha\|\leq Ce^{-\frac{1+\xi}{2}\ell}\|\varphi^\perp\|_*
\end{equation}
for any $\ell$ sufficiently large.
\end{proposition}
\begin{proof}
Recall that
\begin{equation*}
 {\bf r}_\alpha=\left(\begin{array}{c}
\int L(\varphi^\perp)Z_{v,\alpha}\\
\int L(\varphi^\perp) Z_{Y_1,\alpha}^0\\
\vdots\\
\int L(\varphi^\perp)Z_{Y_1,\alpha}^{k-1}\\
\int L(\varphi^\perp)Z_{Y_2,\alpha}^0\\
\vdots\\
\int L(\psi^\perp)Z_{Y_2,\alpha}^{k-1}
\end{array}
\right).
\end{equation*}
The estimate follows from
\begin{equation*}
|\int L(\varphi^\perp)Z_{j,\alpha}^i|\leq Ce^{-\frac{1+\xi}{2}\ell}\|\varphi^\perp\|_*.
\end{equation*}
To prove the above estimate, we fix for example $j=1,i=0,\alpha=3$ and we write
\begin{eqnarray*}
\int L(\varphi^\perp)Z_{1,3}^0&&=\int L(Z_{1,3}^0)\varphi^\perp\\
&&=\int L(\frac{\partial w(x-y_1)}{\partial x_3})\varphi^\perp+L(\frac{\pi_3}{k})\varphi^\perp\\
&&=\int p(|u|^{p-1}-w^{p-1}(x-y_1))\frac{\partial w(x-y_1)}{\partial x_3}\varphi^\perp+O(e^{-\frac{1+\xi}{2}}\ell)\|\varphi^\perp\|_*\\
&&\leq C\int w^{p-2}(x-y_1)|\frac{\partial w(x-y_1)}{\partial x_3}|\sum_{z\in \Pi_{y_1}}w(x-z)+O(e^{-\frac{1+\xi}{2}}\ell)\|\varphi^\perp\|_*\\
&&\leq Ce^{-\frac{1+\xi}{2}}\ell\|\varphi^\perp\|_*\\
\end{eqnarray*}
for some $\xi>0$ independent of $\ell$ large, where we have used the estimate for $\phi$ (\ref{pialpha}),
\begin{equation*}
\|\phi\|_*\leq Ce^{-\frac{1+\xi}{2}\ell}.
\end{equation*}
We proved the estimate for $\alpha=3$. The other cases can be treated similarly.
\end{proof}

\medskip
We have now the tools for the

\noindent{\bf Proof of Proposition \ref{pro303}}. By Proposition \ref{pro401}, we only need to show the following orthogonality conditions:

\begin{equation}\label{eq501}
\left(\begin{array}{c}
{\bf r}_1\\
{\bf r}_2\
\end{array}
\right)\cdot{\bf w}_1=\left(\begin{array}{c}
{\bf r}_1\\
{\bf r}_2\
\end{array}
\right)\cdot{\bf w}_2=\left(\begin{array}{c}
{\bf r}_1\\
{\bf r}_2\
\end{array}
\right)\cdot{\bf w}_3=0
\end{equation}

and
\begin{equation}\label{eq502}
{\bf r}_\alpha\cdot{\bf w}_4={\bf r}_\alpha\cdot{\bf w}_5={\bf r}_\alpha\cdot{\bf w}_6=0.
\end{equation}

First recall that $L(z_1)=0$, then we have
\begin{equation*}
\int L(z_1)\varphi^\perp=\int L(\varphi^\perp)z_1=0.
\end{equation*}
This gives us exactly the first orthogonality condition in (\ref{eq501}). Similarly, from $L(z_2)=0$, one can get that
\begin{equation*}
\int L(z_2)\varphi^\perp=\int L(\varphi^\perp)z_2=0.
\end{equation*}
This gives us exactly the second orthogonality condition in (\ref{eq501}).

For $\alpha=3,\cdots,N$, from $L(z_\alpha)=0$, one can get that
\begin{equation*}
\int L(z_\alpha)\varphi^\perp=\int L(\varphi^\perp)z_\alpha=0.
\end{equation*}
This gives us exactly the first orthogonality condition in (\ref{eq502}).

Second, let us recall that
\begin{equation}
L(\varphi^\perp)=-\sum_{\alpha=1}^N{\bf c}_\alpha \cdot L( {\bf Z}_\alpha ).
\end{equation}
Thus the function $x\to L(\varphi^\perp)(x)$ is invariant under the rotation of angle $\frac{2\pi}{k}$ in the $(x_1,x_2)$ plane. Thus we can get that
\begin{equation*}
\sum_{i=0}^{k-1}(\sum_{j=1}^{m+1}\int L(\varphi^\perp)|y_j|Z_{j,2}^i+\sum_{j=m+2}^{m+2n}\int L(\varphi^\perp)(R_k^iy_{j}\cdot {\bf n}_iZ_{j,1}^i-R_k^iy_j\cdot{\bf t}_iZ_{j,2}^i)=0.
\end{equation*}
This gives us the third orthogonality condition in (\ref{eq501}).

For $\alpha=3,\cdots,N$, we can get that
\begin{equation*}
\sum_{i=0}^{k-1}\sum_{j=1}^{2m+n}\int L(\varphi^\perp)Z_{j,\alpha}^iR_k^iy_j\cdot{\bf e}_1=0,
\end{equation*}
and
\begin{equation*}
\sum_{i=0}^{k-1}\sum_{j=1}^{2m+n}\int L(\varphi^\perp)Z_{j,\alpha}^iR_k^iy_j\cdot{\bf e}_2=0.
\end{equation*}
These give the last two orthogonality conditions in (\ref{eq502}).

Combing the results of Proposition in \ref{pro401} and the a priori estimates in (\ref{eq503}), we get the proof of Proposition \ref{pro303}.

\setcounter{equation}{0}
\section{Final Argument}\label{sec6}
Let $\left(\begin{array}{c}{\bf c}_1\\ \vdots\\ {\bf c}_N\end{array}\right)$ be the solutions to (\ref{eq313}) predicted by Proposition \ref{pro301}, given by
\begin{equation*}
\left(\begin{array}{c}{\bf c}_1\\{\bf c}_2\end{array}\right)=\left(\begin{array}{c}{\bf v}_1\\{\bf v}_2\end{array}\right)+s_1{\bf w}_1+s_2{\bf w}_2+s_3{\bf w}_3,
\end{equation*}
and
\begin{equation*}
{\bf c}_\alpha={\bf v}_\alpha+s_{\alpha1}{\bf w}_4+s_{\alpha2}{\bf w}_5+s_{\alpha3}{\bf w}_6, \ \alpha=3,\cdots,N.
\end{equation*}

A direct computation shows that there exists a unique
\begin{equation*}
(s_1^*,\cdots,s_3^*,s_{31}^*,\cdots,s_{N3}^*)\in \R^{3N-3}
\end{equation*}
for which the above solution satisfies all the $3N-3$ conditions of Proposition \ref{pro301}. Furthermore, one can see that
\begin{equation*}
\|(s_1^*,\cdots,s_3^*,s_{31}^*,\cdots,s_{N3}^*)\|\leq C\ell^\tau\|\varphi^\perp\|_*.
\end{equation*}

Hence there exists a unique solution $\left(\begin{array}{c}{\bf c}_1\\ \vdots\\ {\bf c}_N\end{array}\right)$ to system (\ref{eq313}), satisfying estimates in Proposition \ref{pro303} and the estimate
\begin{equation*}
\sum_{\alpha=1}^N\|{\bf c}_\alpha\|\leq C\ell^\tau e^{\frac{1-\xi}{2}\ell}\|\psi^\perp\|_*.
\end{equation*}
On the other hand, from (\ref{varphi}), we conclude that
\begin{equation*}
\|\psi^\perp\|_*\leq Ce^{-\frac{1+\xi}{2}\ell}\sum_{\alpha=1}^N\|{\bf c}_\alpha\|.
\end{equation*}
Thus we conclude that
\begin{equation*}
c_{j,\alpha}^i=0, \ \psi^\perp=0.
\end{equation*}
This proves Theorem \ref{theorem1}.

\setcounter{equation}{0}
\section{Some Useful Computations}\label{sec7}
In this section, we will perform the computations of the entrance of the matrices $A$, $B$, $C$, $H_\alpha$ for $\alpha=3,\cdots,N$.

We first introduce some useful functions
\begin{equation}
\Psi_1(\ell)=\int div(w^p(x){\bf e})div(w(x-\ell {\bf e}){\bf e})dx,
\end{equation}
\begin{equation}
 \Psi_2(\ell)=\int div(w^p(x){\bf e}^\perp)div(w(x-\ell {\bf e}){\bf e}^\perp)dx,
\end{equation}
where ${\bf e}$ is any unit vector.  It is easy to check that this definition is independent of the choice of the unit vector ${\bf e}$. It is known that
\begin{equation}
\Psi_1(\ell)=C_{N,p,1}e^{-\ell}\ell^{-\frac{N-1}{2}}(1+O(\frac{1}{\ell}))
\end{equation}
and
\begin{equation}
\Psi_2(\ell)=C_{N,p,2}e^{-\ell}\ell^{-\frac{N+1}{2}}(1+O(\frac{1}{\ell}))
\end{equation}
where $C_{N,p,i}>0$ are constants depend only on $p$ and $N$.

In fact, one can see from the definition of $\Psi_1$ and $\Psi$ that
\begin{equation}\label{eq701}
\Psi'_1(\ell)=2\sin\frac{\pi}{k}\Psi_1'(\bar{\ell})\frac{d\bar{\ell}}{d\ell}.
\end{equation}

By these two definitions, one can easily get that
\begin{eqnarray}\label{eq702}
&&\int pw^{p-1}{\bf a}\cdot w(x){\bf b}\cdot w(x-\ell{\bf e})={\bf a}\cdot{\bf e}{\bf b}\cdot{\bf e}\Psi_1(\ell)+{\bf a}\cdot {\bf e}^\perp{\bf b}\cdot{\bf e}^\perp \Psi_2(\ell).\nonumber\\
\end{eqnarray}

\medskip
\noindent
{\bf Computation of $A$}.
\begin{eqnarray*}
\int L(Z_{1,1}^0)Z_{1,1}^0dx&=&\int p(|u|^{p-1}-w^{p-1}(x-y_1))(\frac{\partial w(x-y_1)}{\partial x_1})^2\\
&=&p(p-1)\int w^{p-2}(x-y_1)(\phi+\sum_{z\in \Pi_{y_1}}w(x-z))(\frac{\partial w(x-y_1)}{\partial x_1})^2\\
&+&O(e^{-(1+\xi)\ell}).
\end{eqnarray*}
Recall that $\phi$ solves the following equation:
\begin{equation}
\Delta \phi-\phi+p|U|^{p-1}\phi+E+N(\phi)=0
\end{equation}
where
\begin{equation*}
E=\Delta U-U+|U|^{p-1}U,
\end{equation*}
and
\begin{equation*}
N(\phi)=|U+\phi|^{p-1}(U+\phi)-|U|^{p-1}U-p|U|^{p-1}\phi.
\end{equation*}
Hence we observe that
\begin{eqnarray*}
&&p(p-1)\int w^{p-2}\phi(\frac{\partial w(x-y_1)}{\partial x_1})^2\\
&&=\int \frac{\partial }{\partial x_1}(pw^{p-1}(x-y_1))\phi\frac{\partial }{\partial x_1}w(x-y_1)\\
&&=-\int pw^{p-1}(x-y_1)\frac{\partial}{\partial x_1}(\phi\frac{\partial w(x-y_1)}{\partial x_1})\\
&&=-\int pw^{p-1}(x-y_1)\frac{\partial w(x-y_1)}{\partial x_1}\frac{\partial \phi}{\partial x_1}-\int pw^{p-1}(x-y_1)\phi\frac{\partial^2 w(x-y_1)}{\partial x_1^2}\\
&&=\int[p(|U|^{p-1}-w^{p-1}(x-y_1))\phi+E+N(\phi)]\frac{\partial^2 w(x-y_1)}{\partial x_1^2}\\
&&=\int E\frac{\partial^2 w(x-y_1)}{\partial x_1^2}+O(e^{-(1+\xi)\ell})\\
&&=\int (|U|^{p-1}U-w^p(x-y_1))\frac{\partial^2 w(x-y_1)}{\partial x_1^2}+O(e^{-(1+\xi)\ell})\\
&&=\int pw^{p-1}(x-y_1)(\sum_{z\in\Pi_{y_1}}w(x-z))\frac{\partial^2 w(x-y_1)}{\partial x_1^2}+O(e^{-(1+\xi)\ell}).
\end{eqnarray*}
Taking this into account, we have
\begin{eqnarray*}
\int L(Z_{1,1}^0)Z_{1,1}^0dx&=&p(p-1)\int w^{p-2}(\sum_{z\in \Pi_{y_1}}w(x-z))(\frac{\partial w(x-y_1)}{\partial x_1})^2\\
&+&\int pw^{p-1}(x-y_1)(\sum_{z\in\Pi_{y_1}}w(x-z))\frac{\partial^2 w(x-y_1)}{\partial x_1^2}+O(e^{-(1+\xi)\ell})\\
&=&\int \frac{\partial }{\partial x_1}(pw^{p-1}(x-y_1)\frac{\partial w(x-y_1)}{\partial x_1})\sum_{z\in\Pi_{y_1}}w(x-z))+O(e^{-(1+\xi)\ell})\\
&=&-p\int w^{p-1}(x-y_1)\frac{\partial w(x-y_1)}{\partial x_1}\sum_{z\in\Pi_{y_1}}\frac{\partial w(x-z)}{\partial x_1})+O(e^{-(1+\xi)\ell}).
\end{eqnarray*}

By (\ref{eq702}), we can get that
\begin{eqnarray*}
\int L(Z_{1,\alpha}^0)Z_{1,\alpha}^0dx=-[\Psi_1(\ell)+2(\Psi_1(\bar{\ell})\sin^2\frac{\pi}{k}+\Psi_2(\bar{\ell})\cos^2\theta_0)].
\end{eqnarray*}

Next we consider
\begin{eqnarray*}
\int L(Z_{1,1}^0)Z_{1,1}^1&&=\int (p|u|^{p-1}-pw^{p-1}(x-y_1))\frac{\partial w(x-y_1)}{\partial x_1}R_k^1\cdot \nabla w(x-R_k^1y_1)\\
&&=\int pw^{p-1}(x-R_k^1y_1)\frac{\partial w(x-y_1)}{\partial x_1}R_k^1\cdot \nabla w(x-R_k^1y_1)+O(e^{-(1+\xi)\ell})\\
&&=-\sin^2\frac{\pi}{k}\Psi_1(\bar{\ell})+\cos^2\frac{\pi}{k}\Psi_2(\bar{\ell})+O(e^{-(1+\xi)\ell}).
\end{eqnarray*}
Similarly, one can get that
\begin{eqnarray*}
&&\int L(Z_{1,1}^0)Z_{1,1}^{k-1}=-\sin^2\frac{\pi}{k}\Psi_1(\bar{\ell})+\cos^2\frac{\pi}{k}\Psi_2(\bar{\ell})+O(e^{-(1+\xi)\ell}),\\
&&\int L(Z_{1,1}^0)Z_{i,1}^j=O(e^{-(1+\xi)\ell}) \mbox{ for }(i,j)\neq (1,0),(1,1),(1,k-1),(2,0).
\end{eqnarray*}

Another observation is that
\begin{equation*}
\int L(Z_{i,1}^s)Z_{j,1}^t=\int L(Z_{i,1}^0)Z_{j,1}^{t-s}
\end{equation*}
where we use the notation $Z_{j,1}^{t-s}=Z_{j,1}^{k+t-s}$ if $t-s<0$.

Moreover, for $i\geq 2$,
\begin{equation*}
\int L(Z_{i,1}^s)Z_{j,1}^t=\left\{\begin{array}{l}
-2\Psi_1(\ell)+O(e^{-(1+\xi)\ell}), \mbox{ if }i=j,s=t, i\leq m, \\
\\
\Psi_1(\ell)+O(e^{-(1+\xi)\ell}) \mbox{ if }j=i-1 \mbox{ or }i+1, s=t \ i\leq m,\\
\\
2\Psi_1(\bar{\ell})+O(e^{-(1+\xi)\ell})\mbox{ if }i=j, s=t, \ m+2\leq i\leq 2n+m,\\
\\
\Psi_1(\bar{\ell})+O(e^{-(1+\xi)\ell})\mbox{ if }j=i-1\mbox{ or }i+1, s=t,  \ m+2\leq i\leq 2n+m,\\
\\
-[\Psi_1(\ell)-2(\Psi_1(\bar{\ell})\sin^2\frac{\pi}{k}+\Psi_2(\bar{\ell})\cos^2\frac{\pi}{k})]+O(e^{-(1+\xi)\ell}) \\
\ \ \ \ \ \ \ \ \ \ \ \ \ \ \ \ \ \ \ \ \ \ \ \ \ \ \ \ \ \ \ \ \ \ \ \ \ \ \ \ \ \ \ \ \ \  \ \ \ \ \ \ \ \ \  \mbox{if }i,j=m+1,s=t,\\
\\
\Psi_1(\bar{\ell})\sin\frac{\pi}{k}+O(e^{-(1+\xi)\ell}) \mbox{ if }(i,j)=(m+1,m+2), s=t,\\
\\
-\Psi_1(\bar{\ell})\sin\frac{\pi}{k}+O(e^{-(1+\xi)\ell}) \mbox{ if }(i,j)=(m+1,m+2n), t=s-1,\\
\\
O(e^{-(1+\xi)\ell}) \mbox{ otherwise }.
\end{array}
\right.
\end{equation*}

\medskip
\noindent
{\bf Computation of $C$}: Similarly, we have
\begin{eqnarray*}
\int L(Z_{1,2}^0)Z_{1,2}^0&&=-p\int w^{p-1}(x-y_1)\frac{\partial w(x-y_1)}{\partial x_2}\sum_{z\in\Pi_{y_1}}\frac{\partial w(x-z)}{\partial x_2})+O(e^{-(1+\xi)\ell})\\
&&=-[\Psi_2(\ell)+2(\Psi_1(\bar{\ell})\cos^2\frac{\pi}{k}+\Psi_2(\bar{\ell})\sin^2\frac{\pi}{k})]+O(e^{-(1+\xi)\ell}).
\end{eqnarray*}
and
\begin{equation}
\int L(Z_{1,2}^0)Z_{i,2}^j=\left\{\begin{array}{l}
\Psi_1(\bar{\ell})\cos^2\frac{\pi}{k}-\Psi_2(\bar{\ell})\sin^2\frac{\pi}{k}+O(e^{-(1+\xi)\ell})\\
 \ \ \ \ \ \ \ \ \ \ \ \ \ \ \ \ \ \ \ \ \ \ \ \ \ \ \ \ \ \ \  \ \ \ \ \ \ \ \ \ \mbox{ if }(i,j)=(1,1) \mbox{ or }(1,k-1),\\
 \\
O(e^{-(1+\xi)\ell}) \mbox{ otherwise}.
\end{array}
\right.
\end{equation}

Furthermore, one can get that for $i\geq 2$,
\begin{equation*}
\int L(Z_{i,2}^s)Z_{j,2}^t=\left\{\begin{array}{l}
-2\Psi_2(\ell)+O(e^{-(1+\xi)\ell}), \mbox{ if }i=j,s=t, i\leq m, \\
\\
\Psi_2(\ell)+O(e^{-(1+\xi)\ell}) \mbox{ if }j=i-1 \mbox{ or }i+1, s=t \ i\leq m,\\
\\
2\Psi_2(\bar{\ell})+O(e^{-(1+\xi)\ell})\mbox{ if }i=j, s=t, \ m+2\leq i\leq 2n+m,\\
\\
\Psi_2(\bar{\ell})+O(e^{-(1+\xi)\ell})\mbox{ if }j=i-1\mbox{ or }i+1, s=t,  \ m+2\leq i\leq 2n+m,\\
\\
-[\Psi_2(\ell)-2(\Psi_1(\bar{\ell})\cos^2\frac{\pi}{k}+\Psi_2(\bar{\ell})\sin^2\frac{\pi}{k})]+O(e^{-(1+\xi)\ell}) \\
\ \ \ \ \ \ \ \ \ \ \ \ \ \ \ \ \ \ \ \ \ \ \ \ \ \ \ \ \ \ \ \ \ \ \ \ \ \ \ \ \ \ \ \ \ \  \ \ \ \ \ \ \ \ \ \mbox{if }i,j=m+1,s=t,\\
\\
\Psi_2(\bar{\ell})\sin\frac{\pi}{k}+O(e^{-(1+\xi)\ell}) \mbox{ if }(i,j)=(m+1,m+2), s=t,\\
\\
-\Psi_2(\bar{\ell})\sin\frac{\pi}{k}+O(e^{-(1+\xi)\ell}) \mbox{ if }(i,j)=(m+1,m+2n), t=s-1,\\
\\
O(e^{-(1+\xi)\ell}) \mbox{ otherwise }.
\end{array}
\right.
\end{equation*}

\medskip
\noindent
{\bf Computation of $B$}: Next we consider $\int L(Z_{i,1}^s)Z_{j,2}^t$ and $\int L(Z_{i,2}^s)Z_{j,1}^t$.
First by the symmetry, we have that
\begin{eqnarray*}
\int L(Z_{1,1}^0)Z_{1,2}^0=0,
\end{eqnarray*}
and
\begin{eqnarray*}
\int L(Z_{1,1}^0)Z_{1,2}^1&&=\int pw^{p-1}(x-R_k^1y_1)\frac{\partial w(x-y_1)}{\partial x_1}R_k^{1,\perp}\cdot \nabla w(x-R_k^1y_1)\\
&&=\sin\frac{\pi}{k}\cos\frac{\pi}{k}(\Psi_1(\bar{\ell})+\Psi_2(\bar{\ell}))+O(e^{-(1+\xi)\ell}).
\end{eqnarray*}
Similarly, we can get that
\begin{equation*}
\int L(Z_{1,1}^0)Z_{1,2}^{k-1}=-\sin\frac{\pi}{k}\cos\frac{\pi}{k}(\Psi_1(\bar{\ell})+\Psi_2(\bar{\ell}))+O(e^{-(1+\xi)\ell}),
\end{equation*}
\begin{equation*}
\int L(Z_{m+1,1}^0)Z_{m+2,2}^{0}=\int L(Z_{m+1,1}^0)Z_{2n+m,2}^{k-1}=-\Psi_2(\bar{\ell})\cos\frac{\pi}{k}+O(e^{-(1+\xi)\ell}),
\end{equation*}
and
\begin{equation*}
\int L(Z_{i,1}^s)Z_{j,2}^t=O(e^{-(1+\xi)\ell}) \mbox{ otherwise }.
\end{equation*}
Similarly, we have the following expansion for $\int L(Z_{i,2}^s)Z_{j,1}^t$:
\begin{equation*}
\int L(Z_{i,2}^s)Z_{j,1}^t=\left\{\begin{array}{l}
-\sin\frac{\pi}{k}\cos\frac{\pi}{k}(\Psi_1(\bar{\ell})+\Psi_2(\bar{\ell}))+O(e^{-(1+\xi)\ell}) \mbox{ if }i,j=1, t=s-1,\\
\\
\sin\frac{\pi}{k}\cos\frac{\pi}{k}(\Psi_1(\bar{\ell})+\Psi_2(\bar{\ell}))+O(e^{-(1+\xi)\ell}) \mbox{ if }i,j=1, t=s+1,\\
\\
\Psi_1(\bar{\ell})\cos\frac{\pi}{k}+O(e^{-(1+\xi)\ell})\mbox{ if }(i,j)=(m+2,m+1), s=t \mbox{ or }\\ \ \ \ \ \ \ \ \ \ \ \ \ \ \ \ \ \ \ \ \ \ \ \ \ \ \ \ \ \ \ \ \ \ \ (i,j)=(2n+m,m+1), t=s+1,\\
\\
O(e^{-(1+\xi)\ell})\mbox{ otherwise }.
\end{array}
\right.
\end{equation*}

\medskip
\noindent
{\bf Computation of $H_\alpha$}: For the matrix $H_\alpha$ for $\alpha=3,\cdots, N$, the computation is much easier, we directly use the (\ref{eq702}) to get the following expansion:
\begin{equation*}
\int L(Z_{i,\alpha}^s)Z_{j,\alpha}^t=\left\{\begin{array}{l}
-(\Psi_2(\ell)+2\Psi_2(\bar{\ell}))+O(e^{-(1+\xi)\ell}) \mbox{ if }(i,j)=(1,1),s=t,\\
\\
\Psi_2(\bar{\ell})+O(e^{-(1+\xi)\ell}) \mbox{ if }(i,j)=(1,1),t=s-1 \mbox{ or }s+1,\\
\\
2\Psi_2(\bar{\ell})-\Psi_2(\ell)+O(e^{-(1+\xi)\ell}) \mbox{ if }(i,j)=(m+1,m+1),s=t,\\
\\
-2\Psi_2(\ell)+O(e^{-(1+\xi)\ell}) \mbox{ if }i=j,s=t, 2\leq i\leq m,\\
\\
\Psi_2(\ell)+O(e^{-(1+\xi)\ell}) \mbox{ if }j=i+1\mbox{ or }i-1,s=t, 2\leq i\leq m,\\
\\
2\Psi_2(\bar{\ell})+O(e^{-(1+\xi)\ell}) \mbox{ if }i=j,s=t, m+2\leq i\leq m+2n,\\
\\
\Psi_2(\bar{\ell})+O(e^{-(1+\xi)\ell}) \mbox{ if }j=i+1\mbox{ or }i-1,s=t, m+2\leq i\leq m+2n,\\
\\
O(e^{-(1+\xi)\ell})\mbox{ otherwise}.
\end{array}
\right.
\end{equation*}

\end{document}